\documentclass[11pt]{article}		
\usepackage[utf8]{inputenc}	
\usepackage[T1]{fontenc}	
\usepackage{setspace}

\usepackage{fancyhdr}
\usepackage{amsmath,amssymb,amsbsy,amsfonts,latexsym, amsthm,
         amsopn,amstext,amsxtra,euscript,amscd, mathtools, dsfont, fullpage}
         \usepackage{amsthm}
\usepackage{tikz}
\usepackage{pgf}
\usepackage{enumitem}
\usepackage{physics}
\usepackage{cprotect}	
\newcommand{\inner}[2]{\left\langle #1, #2 \right\rangle} 
\usepackage{hyperref} 
\usepackage[capitalise]{cleveref}
\usepackage{wrapfig}
\usepackage{mathrsfs}

\definecolor{ffqqqq}{rgb}{1,0,0}
\DeclareMathOperator{\codim} {codim}

\numberwithin{equation}{section}

\newtheorem*{corollary*}{Corollary}
\newtheorem*{theorem*}{Theorem}

\newtheorem{theorem}{Theorem}[section]
\newtheorem{thm}[theorem]{Theorem}
\newtheorem{corollary}[theorem]{Corollary}
\newtheorem{cor}[theorem]{Corollary}
\newtheorem{proposition}[theorem]{Proposition}

\newtheorem{lem}[theorem]{Lemma}
\newtheorem{conjecture}[theorem]{Conjecture}
\newtheorem{question}[theorem]{Question}
\newtheorem*{question*}{Question}
\newtheorem*{Tarski's plank problem}{Tarski's plank problem}

\theoremstyle{definition}

\theoremstyle{remark}

\newtheorem{rem}[theorem]{Remark}

\definecolor{qqqqff}{rgb}{0,0,1}
\definecolor{qqwuqq}{rgb}{0,0.39215686274509803,0}
\definecolor{zzttff}{rgb}{0.6,0.2,1}
\definecolor{ccqqqq}{rgb}{0.8,0,0}
\definecolor{ffvvqq}{rgb}{1,0.3333333333333333,0}

\title{Plank theorems and their applications: a survey
}

\author{William Verreault}

\date{}

\usepackage[nottoc]{tocbibind} 

\usepackage{amsmath,amsthm,amssymb,amsfonts}	
\usepackage{graphicx}	
\usepackage{color}	
\usepackage{units}	
\usepackage{pgf,tikz}	
\usetikzlibrary{arrows} 
\usepackage{hyperref}	
\newcommand{\N}{\mathbb{N}}	
\newcommand{\Z}{\mathbb{Z}}	
\newcommand{\R}{\mathbb{R}}	
\newcommand{\C}{\mathbb{C}}
\usepackage{cprotect}	

\begin{document}
\maketitle

\thispagestyle{fancy}

\begin{abstract}
Plank problems concern the covering of convex bodies by planks in Euclidean space and
are related to famous open problems in convex geometry. In this survey, we introduce plank
problems and present surprising applications of plank theorems 
in various areas of
mathematics.
\end{abstract}

\tableofcontents

\section{Introduction} \label{sec:intro} 

Questions about the covering of convex bodies by planks in Euclidean space are referred to as \textit{plank problems} in convex geometry and discrete geometry. Here is a toy plank problem.

\begin{question} \label{question}
Given a circular table of diameter $d$ and planks of width $1$ and length at least $d$, what is the minimal number of planks needed to cover the table?
\end{question}
It is readily seen that we may use $d$ parallel planks, but can we do better by changing their orientation and by overlapping them if needed? The answer is no. In fact, we will see that even if we consider a countable family of planks of positive and varying widths, it is impossible to cover the table in a way such that the sum of the widths of the planks used is less than $d$. Accordingly, a covering by parallel planks is optimal. 
Alfred Tarski generalized this question in the 1930s, a period when discrete geometry was still nascent, 
and what became known as \textit{Tarski's plank problem} is still the source of numerous problems and conjectures in geometry despite having been solved decades ago.

\begin{figure}[ht]
\centering
\begin{tikzpicture}[scale=2]

\draw[thick] (-1,-1.1) -- (-1,1.1) ;
 \draw[dashed] (-1,-1.1) -- (-1,-1.5) ;
 \draw[dashed] (-1,1.1) -- (-1,1.5) ;

\draw[thick] (-0.67,-1.1) -- (-0.67,1.1);
 \draw[dashed] (-0.67,-1.1) -- (-0.67,-1.5);
 \draw[dashed] (-0.67,1.1) -- (-0.67,1.5);

\draw[thick] (-0.34,-1.1) -- (-0.34,1.1);
 \draw[dashed] (-0.34,-1.1) -- (-0.34,-1.5);
 \draw[dashed] (-0.34,1.1) -- (-0.34,1.5);

\draw[thick] (0,-1.1) -- (0,1.1);
 \draw[dashed] (0,-1.1) -- (0,-1.5);
 \draw[dashed] (0,1.1) -- (0,1.5);

\draw[thick] (0.34,-1.1) -- (0.34,1.1) ;
 \draw[dashed] (0.34,-1.1) -- (0.34,-1.5) ;
 \draw[dashed] (0.34,1.1) -- (0.34,1.5) ;

\draw[thick] (0.67,-1.1) -- (0.67,1.1) ;
 \draw[dashed] (0.67,-1.1) -- (0.67,-1.5) ;
 \draw[dashed] (0.67,1.1) -- (0.67,1.5) ;

\draw[thick] (1,-1.1) -- (1,1.1) ;
 \draw[dashed] (1,-1.1) -- (1,-1.5) ;
 \draw[dashed] (1,1.1) -- (1,1.5) ;


\draw[thick] (0,0) circle (1) ;

\draw [thick] (3,0) circle (1);
\draw [thick] (1.9287565953731889,-0.29537871958690615)-- (3.898968427006085,0.6835315615388711);
\draw [thick] (2.066017166308294,-0.6048837390141275)-- (4.035605919416572,0.3752795859625725);
\draw [thick] (2.0458192515073543,0.5272237289234458)-- (3.929783500003386,-0.6088566865352416);
\draw [thick] (2.2324867302077807,0.7929875969037133)-- (4.130292299480319,-0.3198169776906491);
\draw [thick] (2.0426553959361606,0.628467107201643)-- (4.1913420377784,0.15608408975846438);
\draw [thick] (2.1249156407871963,0.9416888087498155)-- (4.273602282629435,0.46930579130663685);
\draw [thick] (2.469775898047306,-1.073687190100547)-- (4.153822882364229,0.3419348479495392);
\draw [thick] (2.245142152492556,-0.8300703111186349)-- (3.9291891368094785,0.5855517269314512);
\draw [thick,dashed] (1.5712446265338567,-0.47628177308563896)-- (1.9284599660861548,-0.2962895477298299);
\draw [thick,dashed] (1.7083565811816546,-0.7840841202541647)-- (2.065571920733953,-0.6040918948983558);
\draw [thick,dashed] (1.9395168427813525,-1.083980419063509)-- (2.246219911978594,-0.8272057564797717);
\draw [thick,dashed] (4.033965570921244,0.3765298144387211)-- (4.391180910473542,0.5565220397945301);
\draw [thick,dashed] (3.9021478555464233,0.6849725366865264)-- (4.259363195098722,0.8649647620423353);
\draw [thick,dashed] (3.9288928922520183,0.5866751717771714)-- (4.23559596144926,0.8434498343609087);
\draw [thick,dashed] (4.156550362702006,0.3450080108379535)-- (4.463253431899248,0.6017826734216907);
\draw [thick,dashed] (2.163671890609033,-1.3291500026250345)-- (2.4703749598062745,-1.0723753400412974);
\draw [thick,dashed] (4.1320334043458535,-0.3204522874004725)-- (4.471708288875222,-0.5316890569742385);
\draw [thick,dashed] (1.7013521044644417,0.7442841897810095)-- (2.039383638885328,0.5304275047392247);
\draw [thick,dashed] (1.893985348691355,1.01046830907638)-- (2.2311542325223903,0.7952541279076342);
\draw [thick,dashed] (2.048091944072885,0.6287042432448618)-- (1.6574492044500053,0.7147173235287987);
\draw [thick,dashed] (4.272130309238154,0.471095225241024)-- (4.66154153980651,0.3796682406728009);
\draw [thick,dashed] (4.195077011547389,0.15587718923334842)-- (4.5844882421157465,0.06445020466512576);
\draw [thick,dashed] (2.1321500870082657,0.9439222792525374)-- (1.7421642655979075,1.0328664139601629);
\draw [thick] (1.9465216880259673,-1.0104295439950512)-- (2.4393377623801373,1.1336628786533163);
\draw [thick] (2.2722469919005657,-1.069970728574279)-- (2.7650630662547355,1.0741216940740885);
\draw [thick,dashed] (2.2722469919005657,-1.069970728574279)-- (2.184442113851785,-1.4602146310133035);
\draw [thick,dashed] (1.9465216880259673,-1.0104295439950512)-- (1.8587168099771867,-1.4006734464340758);
\draw [thick,dashed] (2.853053786011319,1.4650623906745635)-- (2.7652489079625386,1.0748184882355392);
\draw [thick,dashed] (2.5269787874081517,1.522847833464998)-- (2.439173909359371,1.1326039310259737);
\draw [thick,dashed] (3.9138294143785872,-0.6028309548974209)-- (4.253252884114923,-0.8144714713212535);
\draw [thick,dashed] (3.4515458720551084,1.0069063799789741)-- (3.641925790731928,1.3586953601426621);
\draw [thick] (3.4515458720551084,1.0151614432347504)-- (2.3969077246690063,-0.915573747846101);
\draw [thick] (3.7446006176351707,0.8624427730028873)-- (2.6871692667195415,-1.0667640427897451);
\draw [thick,dashed] (3.748728149263059,0.8624427730028873)-- (3.9391080679398787,1.2142317531665752);
\draw [thick,dashed] (2.683824989267903,-1.0692420288487867)-- (2.497369603559341,-1.42312674049973);
\draw [thick,dashed] (2.3948977753157292,-0.9165233586169236)-- (2.206484700987875,-1.2693696614490864);







\end{tikzpicture}
\caption{Optimal and nonoptimal positioning of six unit planks on a circular table.} \label{fig:1}
\end{figure}


This survey covers a few plank problems and plank theorems and presents various applications of these theorems and the ideas behind their proofs in different fields of mathematics. It is organized as follows. Tarski's plank problem is described in \cref{sec:2}. A solution and a generalization from Th{\o}ger Bang, as well as an important result called Bang's lemma, are also included in this section.
We present a famous plank theorem of Keith Ball for Banach spaces in \cref{sec:3}, which proves a specific but important case of Bang's conjecture. 
In \cref{sec:app1}, we move on to discussing different ways to strengthen Bang's lemma and Ball's plank theorem, as well as several other types of plank problems: spherical and discrete variants, and plank problems on complex Hilbert spaces and $L_p(\mu)$ spaces. We also present a recent polynomial approach to plank coverings, which is based on general results regarding the zeros of polynomials restricted to the unit sphere and which has been used to obtain streamlined demonstrations of some plank theorems. We then highlight the connection between plank problems and diverse fundamental concepts in analysis and number theory in Sections \ref{sec:app2} and \ref{sec:app3}. 
The reader can refer to the Table of Contents for a more thorough description of the applications.



This survey is by no means exhaustive,
the aim of this paper being first and foremost to present the main ideas and applications of plank theorems. The author still tried to include the major parts of the history behind these problems as well as many details in the proofs of the important results for completeness and an easy access to all readers; the main reason for this being that there is no recent comprehensive survey on
plank problems, except possibly \cite{springerbook} whose focus on Tarski's plank problem is very brief.
Anyone familiar with plank theorems is invited to skip over to the applications. 


\subsection{Notation and conventions} Unless stated otherwise, we work in $\R^d$, and it is always assumed that $d\geq 2$.
The letter $C$ will  denote a $d$-dimensional \textit{convex body}, that is, a compact convex subset of $\R^d$ with nonempty interior, while $H$ will denote a \textit{hyperplane} in $\R^d$,
that is, an affine subspace of dimension $d-1$. 

We recall a few standard definitions needed to properly state and understand plank problems.
A (closed) \textit{plank} $P$ in $\R^d$ is the closed set of points between two parallel hyperplanes that are called the \textit{boundary hyperplanes} of $P$. We say that a hyperplane \textit{supports} a convex body $C\subset \R^d$ if $C$ is entirely contained in one of the half-spaces bounded by the hyperplane and the boundary of $C$ touches the hyperplane in at least one point. Once we fix a direction, we can talk about \textit{the} hyperplanes that support $C$ (in the given direction) without ambiguity. Finally, we say that planks $P_1,P_2,\ldots, P_n $  \textit{cover} $C$ if $C\subseteq\bigcup_{i=1}^{n} P_i$. In most problems we shall encounter, it does not matter that the number of planks be finite or countable, but we will assume that it is finite for simplicity's sake. We refer to the planks $P_i$ as a  \textit{covering} of $C$.
The \textit{width $w(P)$ of a plank} $P$ is the distance between the two boundary hyperplanes of $P$. We sometimes refer to the sum of the widths of a covering of a convex body as the \textit{total width} of the covering or of the associated planks. 
Given a hyperplane $H$, the \textit{width $w(C,H)$ of a convex body $C$ in the direction of $H$} is the distance between the two hyperplanes that support $C$ and that are parallel to $H$.
We also define the \textit{minimal width $w(C)$ of a convex body} $C$ as the infimum of $w(C,H)$ over all hyperplanes $H\subset \R^d$.

\begin{figure}[ht]
\centering
\begin{tikzpicture}[scale=1]
    \draw[rotate=45] (0,0) ellipse (2cm and 1cm);
    \draw (1.1,1.8) node {$C$} ;
    
    
    \draw[thick] (3.67,-1.9) -- (3.67,1.9);
    \draw[dashed] (3.67,-1.9) -- (3.67,-2.6);
    \draw[dashed] (3.67,1.9) -- (3.67,2.6);
     \draw (3.9,0) node {$H$} ;
    
    \draw[thick] (1.6,-1.9) -- (1.6,1.9) ;
    \draw[dashed] (1.6,-1.9) -- (1.6,-2.6) ;
    \draw[dashed] (1.6,1.9) -- (1.6,2.6) ;
    
    \draw[thick] (-1.6,-1.9) -- (-1.6,1.9) ;
    \draw[dashed] (-1.6,-1.9) -- (-1.6,-2.6) ;
    \draw[dashed] (-1.6,1.9) -- (-1.6,2.6) ;
    
    \draw[] (-1.6,-1.65) -- (1.6,-1.65) ;
    \draw (0,-1.65) node[below] {$w(C,H)$} ;

\end{tikzpicture}
\caption{Example of the width of a convex body in the direction of a hyperplane in $\R^2$.} 
\end{figure}


\section{Tarski's plank problem} \label{sec:2}

\subsection{The genesis of plank problems}
In the 1930s, the mathematician and logician Alfred Tarski, who is known for his work on model theory and the Banach--Tarski paradox, among other things, proposed a problem that would change the face of discrete geometry. It is what we now call a plank problem.

\begin{Tarski's plank problem} \label{conj Tarksi} If $C\subset \R^d$ is covered by a sequence of planks $P_1,P_2\ldots,P_n$, then the sum of the widths of the planks is at least $w(C)$.
\end{Tarski's plank problem}
Without loss of generality, we may consider a body of minimal width $1$.
Then, similarly to our answer to \cref{question}, it is obvious that we can cover a convex body of minimal width $1$ with planks that have total width $1$, 
by placing a family of parallel planks perpendicular to a pair of supporting hyperplanes that realize the minimal width.
But can we do better? Tarski's plank problem says no.


\subsubsection{Partial solution to Tarski's problem} \label{sec:partialsol}
While the statement of the plank problem above never explicitly appeared in Tarski's work, Bang attributed the two-dimensional version of this conjecture to Tarski in \cite{bang1950}. 
In fact, Tarski proved his conjecture for figures for which the width equals the diameter of the largest inscribed disk \cite{Tarksi}.
While this applies to figures like disks and parallelograms, as observed by Tarski,  the case of $\R^2$ is not entirely covered by this partial result, since some convex figures, for example an equilateral triangle, have a width bigger than the diameter of their inscribed circle; yet, Tarski's arguments generalize to three-dimensional solids whose width equals the diameter of the largest inscribed ball. However, they do not work in higher dimensions because the proof relies on geometric characteristics of two (or three)-dimensional spaces. His proof is still interesting in its own right, so we include it 
by adapting the argument that was presented in
 \cite{kuppach}. Note that Tarski was inspired by a solution of Moese \cite{moese} on a related problem first stated by Tarksi himself (read \cite[Chapter 7]{Tarskitrans} for a discussion on the history of this problem and a translation of the related papers of Tarski and Moese).

The demonstration relies essentially on a result of Archimedes in \textit{On the Sphere and Cylinder}, often called ``Archimedes' Hat-Box Theorem'', which can be formulated as follows.


\begin{proposition} \label{Archimede}
The lateral area of a spherical segment formed by the intersection of a sphere of radius $r$ with the region between two parallel hyperplanes separated by distance $d$, both of which intersect the sphere,
is $2\pi rd$.
\end{proposition}

We can now use Archimedes' result to outline Tarski's argument.

\begin{proof}[Proof of a special case of Tarski's problem]
Let $D$ and $B$ denote the disk and the sphere of radius $r$ centered at the origin, respectively.
Suppose that planks $P_1,P_2,\ldots, P_n$ cover $D$ and that each plank is entirely contained in $D$ (replacing $P_i$ with $P_i\cap D$ if needed).
Then, the union of the vertical projections of the planks on $B$ covers the surface of $B$, which has area $4\pi r^2$. Thus, the sum of the lateral areas of the spherical segments formed by the intersection of the vertical projections of the $P_i$ with $B$ is at least $4\pi r^2$. By \cref{Archimede}, it follows that 
$$
\sum_{i=1}^n2\pi r w(P_i) \geq 4\pi r^2,
$$
whence $\sum_{i=1}^n w(P_i) \geq 2r = w(D)$.
\end{proof}


\subsection{Bang's solution and conjecture}
In 1950, Bang \cite{bang1950} gave an unexpectedly short proof of Tarski's conjecture. 
His solution relies heavily on simple geometric ideas, but they are hiding the combinatorial dimension of the argument. 
The basic ideas behind the proof are still at the core of many demonstrations in convex geometry and discrete geometry, which we shall partly explore in Sections \ref{sec:app1} to \ref{sec:app3}. 
One year later, Bang \cite{bang1951} simplified his proof and Fenchel \cite{fenchel} simplified it some more. Finally in 1964, Bogn\'ar \cite{bognar} reformulated the demonstration of Fenchel. Bogn\'ar's trick of reducing the statement to the case of centered planks is still sometimes used to simplify plank problems (see \cref{sec:dig}).

\begin{thm} \label{Thm Bang} Let $C\subset \R^d$ be a convex body covered by $n$ planks $P_j$, $j=1,\ldots,n$. Then ${w(P_1)+\cdots+w(P_n)\geq w(C)}$.
\end{thm}

The proof we include here mostly follows another reformulation \cite{Nazarov}, but it is more detailed and also presents the end of the demonstration in terms of the important \textit{Bang lemma}, which will be used extensively in the rest of this paper.

\subsubsection{A solution to Tarski's plank problem}
In what follows, we let $a_j=w(P_j)/2$ and we use $u_j$ to denote a vector of length $a_j$ that is perpendicular to the boundary hyperplanes of $P_j$ (sometimes called the \textit{width vector} of $P_j$). The whole proof essentially relies on two lemmas and then proceeds by contradiction. Here is the first lemma.
\begin{lem} \label{lem1Bang}
Let $u$ be a vector of norm $a/2$ with $a<w(C)$,
and let $m$ be the midpoint of a longest chord in $C$ parallel to $u$ (if several longest chords exist, pick any one of them). 
Then, 
the intersection $(C-u)\cap(C+u)$ contains a homothetic image of $C$ with center of homothety $m$ and homothety ratio 
$\kappa=(w(C)-a)/w(C)$.
\end{lem}

In the above, 
$C\pm u$ is the translation of $C$ by the vector $\pm u$.
As this is a simple geometric statement, one may be convinced of its veracity by looking at \cref{fig:hom}. We still give a full proof below for completeness. Also, if the longest chord that is parallel to 
$u$ is not unique, the midpoint 
$m$ is not uniquely determined, but the lemma holds for any such choice of chord, and the subsequent corollaries only need this type of existence statement.

\begin{figure}[ht]
\centering
\begin{tikzpicture}[scale=1]

\draw [rotate around={-45:(2,2)},line width=1.5pt] (2,2) ellipse (2.307825776141003cm and 1.8237488349607849cm);
\draw [rotate around={-45:(1.3339692554973284,1.2442258647289295)}] (1.3339692554973284,1.2442258647289295) ellipse (2.307825776141003cm and 1.8237488349607849cm);
\draw [rotate around={-45:(2.8339692554973155,2.9142258647289427)}] (2.8339692554973155,2.9142258647289427) ellipse (2.307825776141003cm and 1.8237488349607849cm);
\draw [rotate around={-45:(2.070470903750556,2.0852409015825146)},dashed, line width=1.5pt] (2.070470903750556,2.0852409015825146) ellipse (0.7615825061265304cm and 0.6018371155370587cm);
\draw [->] (4.15141635066857,-0.06718209422865806) -- (4.994714147567515,0.8454278503606143);
\begin{scriptsize}
\draw (0.2,3.9) node {$C$};
\draw (-1.4,1.6) node {$C-u$};
\draw (3.9,4.9) node {$C+u$};
\draw(2.07,2.1) node {$\kappa\cdot C$};
\draw (4.740569606036327,0.33136275499070794) node {$u$};
\end{scriptsize}
\end{tikzpicture}
\caption{An example from \cref{lem1Bang}. The intersection $(C-u)\cap(C+u)$ contains $\kappa\cdot C$.}  \label{fig:hom}
\end{figure}

\begin{proof}[Proof of \cref{lem1Bang}]
Let the chosen longest chord in $C$ parallel to $u$ have endpoints $y$ and $z$, and length $\ell$. (If there are several longest chords parallel to
$u$, choose any one; the argument below does not change.) Under these hypotheses, we may write $u=a(z-y)/2\ell$. If $x\in C$, then the point
$$
\frac{a}{\ell}m+\frac{\ell-a}{\ell}x + u =\frac{a}{\ell}\frac{y+z}{2}+\frac{\ell-a}{\ell}x+\frac{a}{2}\frac{z-y}{\ell} = \frac{a}{\ell}z+\frac{\ell-a}{\ell}x
$$
is in $C$, since $a/\ell+(\ell-a)/\ell=1$ and $C$ is convex.
The same reasoning shows that 
$ \displaystyle
\frac{a}{\ell}m+\frac{\ell-a}{\ell}x - u
$
is also in $C$, 
whence $$\frac{a}{\ell}m+\frac{\ell-a}{\ell}x \in (C-u)\cap(C+u).$$ 
Under the homothety of center $m$ and ratio $(\ell-a)/\ell$, the convex body $C$ is contracted and covered by the set of points of the form $\displaystyle \frac{a}{\ell}m+\frac{\ell-a}{\ell}x.$ We conclude that the image of $C$ under this homothety is contained in $(C-u)\cap(C+u)$.

Let us now compare $\ell$ and $w(C)$. There exist supporting hyperplanes to $C$ at $y$ and $z$ whose normals are parallel to the chord direction $z-y$, else we could move the chord to obtain a bigger one, contradicting maximality. 
Since the minimal length of $C$ is defined as the infimum over the distance between such hyperplanes, we must have $\ell\geq w(C)$. It follows that $(\ell-a)/\ell \geq (w(C)-a)/w(C)$, from which we see that the image of $C$ under the homothety with center $m$ and ratio $(w(C)-a)/w(C)$ is contained in the image of $C$ under the homothety with center $m$ and ratio $(\ell-a)/\ell$, which is contained in $(C-u)\cap(C+u)$ as seen above.
\end{proof}

\cref{lem1Bang} might seem to have nothing to do with \cref{Thm Bang}, but we shall rather care about a corollary of this lemma. Note that by successive applications of the preceding lemma on the vectors $u_j$ (remembering that we defined $a_j=w(P_j)/2$ as the length of $u_j$), we get that the intersection over all $2^n$ possible choices of signs $\epsilon_j\in \{\pm 1\}$,
\begin{equation} \label{eq:int}
\bigcap_{\epsilon_j \in \{\pm 1\}} \Big(C-\sum_{j=1}^n \epsilon_ju_j\Big),
\end{equation}
must contain a homothetic image of $C$ with a homothety ratio equal to
$$
\frac{w(C)-\sum_{i=1}^n w(P_i)}{w(C)}.
$$
Thus, under the assumption $w(P_1)+\cdots+w(P_n)<w(C)$, we can find $x_0$ in the intersection \eqref{eq:int} such that $x_\epsilon:=x_0+\sum_{j=1}^n \epsilon_ju_j$ is in $C$ for all choices of sign $\epsilon=(\epsilon_j)_{j=1}^n$. In conjunction with the following lemma, this will lead us to a contradiction.

\begin{lem} \label{Lem2Bang}
For all $x_0\in\R^d$, there is a choice of signs $\epsilon$ such that $x_\epsilon\notin P_j$ for $j=1,2\ldots,n$.
\end{lem}

Rather than proving this lemma right away, let us consider another related result that will prove essential in completing this argument and in the following sections of this survey. We will refer to it as Bang's lemma. The proof follows \cite{ball1991} and \cite{ballconvex}.

\begin{lem}[Bang's lemma] \label{lemBangV2} Let $(u_i)_{i=1}^n$ be unit vectors in $\R^d$, $(m_i)_{i=1}^n$ real numbers, and $(w_i)_{i=1}^n$ nonnegative real numbers. Then there is a choice of signs $(\epsilon_i)_{i=1}^n$ such that the vector $x=\sum_{j=1}^n \epsilon_jw_ju_j$ satisfies 
$$
|\inner{x}{u_i}-m_i|\geq  w_i
$$
for all $1\leq i\leq n$.
\end{lem}

Observe that we can reformulate this lemma in the language of linear algebra by introducing a matrix composed of all the possible scalar products between the vectors $u_i$, the \textit{Gram matrix of $(u_i)_{i=1}^n$}. Explicitly, define $H$ by $h_{ij}=\inner{u_i}{u_j}$. Then, \cref{lemBangV2} is equivalent to the existence of a choice of signs $(\epsilon_i)_{i=1}^n$ such that, for all $i$,
$$
\Bigl|\sum_{j=1}^n h_{ij}\epsilon_jw_j - m_i\Bigr| \geq w_i.
$$

We refer the reader to Subsection \ref{sec:dig} for a geometric approach to Bang's lemma and a justification for the choice of the maximizing vector $(\epsilon_i)_{i=1}^n$, both of which are due to Bogn\'ar.

\begin{proof}[Proof of \cref{lemBangV2}]
Choose the $\epsilon_i\in \{\pm 1\}$ such that the difference 
$$\sum_{i,j=1}^n h_{ij}\epsilon_i\epsilon_jw_iw_j - 2\sum_{i=1}^n\epsilon_iw_im_i$$
is maximized. We shall show that this choice of signs satisfies the statement of the lemma. For  $1\leq k\leq n$, let
$$
(\delta_i)_{i=1}^n = (\epsilon_1,\ldots,\epsilon_{k-1},-\epsilon_k,\epsilon_{k+1},\ldots,\epsilon_n),
$$
that is, flip the sign of the $k$th term in the sequence $(\epsilon_i)_{i=1}^n$.
Then,
$$
\sum_{i,j=1}^n h_{ij}\epsilon_i\epsilon_jw_iw_j - 2\sum_{i=1}^n\epsilon_iw_im_i \geq 
\sum_{i,j=1}^n h_{ij}\delta_i\delta_jw_iw_j -2\sum_{i=1}^n \delta_iw_im_i
$$
for each $k$. By symmetry, we can reorganize into the inequality
\begin{align*}
    0&\leq \sum_{i,j=1}^n h_{ij}\epsilon_i\epsilon_jw_iw_j - 2\sum_{i=1}^n\epsilon_iw_im_i -\Bigl(\sum_{i,j=1}^n h_{ij}\delta_i\delta_jw_iw_j -2\sum_{i=1}^n \delta_iw_im_i\Bigr) \\
    &=4\epsilon_kw_k\sum_{j\neq k}h_{kj}\epsilon_jw_j -4\epsilon_kw_km_k
    \\
    &= -4w_k^2 + 4\epsilon_kw_k\sum_{j=1}^n h_{kj}\epsilon_jw_j -4\epsilon_kw_km_k,
\end{align*}
where the last equality is justified by the fact that $h_{kk}=1$. Dividing by $4w_k$ and reorganizing, we obtain
$$ 
w_k \leq \epsilon_k\Bigl(\sum_{j=1}^n h_{kj}\epsilon_jw_j-m_k\Bigr) \leq \Bigl|\sum_{j=1}^n h_{kj}\epsilon_jw_j-m_k\Bigr|. \qedhere
$$
\end{proof}

We could ask that the inequality in Bang's lemma be strict by using a compactness argument (that is, by slightly modifying the width of our planks if we want to work with open ones). This general argument implies that in plank theorems, we could equivalently work with closed or open planks, and it is explained further in the paragraph following \cref{thm Ball}.

We are now ready to prove \cref{Lem2Bang}.
\begin{proof}[Proof of \cref{Lem2Bang}]
Fix $1\leq k\leq n$ and let $p_k$ be a point of the middle (or central) hyperplane of $P_k$. Note that if $z$ belongs to one of the boundary hyperplanes of $P_k$, then $\abs{\inner{z-p_k}{u_k}}=a_k^2$, so that $P_k$ is given explicitly by
$$
\{x\in \R^d : \abs{\inner{x-p_k}{u_k}}\leq a_k^2\}.
$$
To find an $x_\epsilon$ that is not in $P_k$, we are looking for $\epsilon$ such that 
$
\abs{\inner{x_\epsilon-p_k}{u_k}}> a_k^2,
$
or equivalently, such that
$$\abs{\inner{\frac{x_\epsilon-x_0}{a_k}}{\frac{u_k}{a_k}}-\frac{1}{a_k^2}\inner{p_k-x_0}{u_k}}> 1.$$
Since $u_k/a_k$ are unit vectors and
$$\frac{x_\epsilon-x_0}{a_k}=\sum_{k=1}^n\epsilon_k\frac{u_k}{a_k},$$
letting $m_k:=\inner{p_k-x_0}{u_k}/a_k^2$ and $w_k:=1$, \cref{lemBangV2} allows us to conclude (here we implicitly used the remark made in the paragraph following the proof of \cref{lemBangV2}). Since $k$ was arbitrary, we are done.
\end{proof}

It is now easy to finish the proof of \cref{Thm Bang}. \cref{Lem2Bang} informs us that there is a point which is in $C$ but in none of the planks that cover $C$, a contradiction. Whence it must be that $w(P_1)+\cdots+w(P_n)\geq w(C)$.

\subsubsection{A digression on Bogn\'ar's trick} \label{sec:dig}
Recall that \cref{Lem2Bang} says that in all the translates of the set 
\begin{equation*}
    \Bigl\{\sum_{j=1}^n \epsilon_ju_j : \epsilon_j\in \{\pm 1\}\Bigr\},
\end{equation*} 
there exists a point which 
is not covered by any of the planks that cover the convex body $C$, where $u_j$ is the width vector of the $j$th plank (as in the proof of \cref{Thm Bang}). Likewise, put simply in terms of plank problems, Bang's lemma implies that there is a choice of signs $(\epsilon_i)_{i=1}^n$ such that the point $\sum_{j=1}^n \epsilon_jw_ju_j$ does not belong to any of the open planks \begin{equation} \label{eq:plankbang}
P_i:=\{x\in \R^d : |\inner{x}{u_i}-m_i|< w_i\}
\end{equation}
that cover $C$.
The proof given above relies on the optimization of a well-chosen quadratic function on $\{\pm 1\}^n$. 
On the other hand, Bogn\'ar's \cite{bognar} alternative proof of \cref{Thm Bang} relies on the observation that it suffices to prove \cref{lemBangV2} for \textit{centered} planks, which is sometimes referred to as Bogn\'ar's trick and is still used to simplify the solutions of some discrete geometry problems today (see, for example, the recent papers \cite{avoidingzeros}, \cite{complnou}, \cite{LFT}).
Here, by centered planks we mean that their middle (or central) hyperplanes have a common point, or equivalently via a shift, that $m_i=0$ for all $i$.
But the proof of \cref{lemBangV2} in this case is almost immediate, since if for any $x_0\in\R^d$ we pick $x':=x+x_0$ with maximal norm (where $x$ is as in the lemma), then 
$$
0\leq \|x'\|^2-\|x' - 2\epsilon_jw_ju_j\|^2 = 4\epsilon_jw_j\inner{x'}{u_j} -4w_j^2,
$$
which is equivalent to $|\inner{x'}{u_j}|\geq w_j$. Note that $x' - 2\epsilon_jw_ju_j$ appeared implicitly in the previous proof of Bang's lemma when we flipped the sign of the $k$th term in the sequence $(\epsilon_i)_{i=1}^n$. 

To see why it suffices to consider centered planks, we partially follow \cite{amb}. Write $H_i$ for the middle hyperplane of the plank $P_i$ given by \eqref{eq:plankbang}, think of $\R^d$ as being embedded in $\R^{d+1}$ as a coordinate subspace
$$
H:=\{(x,0)\in \R^{d+1} : x\in \R^d\}
$$
with origin unaltered in $\R^d$, and let $e$ be a unit vector normal to $H$. If $p_t$ stands for the point $t\cdot e$ $(t>0)$, define hyperplanes $H_i^t$ by $p_t\in H_i^t$ and $H_i\subset H_i^t$, as well as planks $P_i^t$ with middle hyperplane $H_i^t$ by $P_i^t\cap H = P_i$. By construction, the $P_i^t$ have a common point $p_t$, so we can apply the special case that we have just discussed to deduce that for all $t>0$, there is a choice of signs $(\epsilon_j^{(t)})_{j=1}^n$ and a point
$$
y^{(t)} = \sum_{j=1}^n \epsilon_j^{(t)}w_ju^t_j 
$$
which is in none of the $P_i^t$, where $u_j^t$ is the width vector of $P_j^t$.
Writing $y^{(t)}=p_t+v^{(t)}$ for $v^{(t)}\in H$, the sequence $(v^{(t)})$ has a convergent subsequence since its norm is uniformly bounded, say $v^{(t_k)}\to v\in H$. Because there are only finitely many possible sign vectors, we may pass to a further subsequence where it is constant, say given by $(\epsilon_j)_{j=1}^n$. Because $u_j^t\to u_j$ as $t\to\infty$, the limit point $v$ has the form 
$$
v=\sum_{j=1}^n \epsilon_j w_j u_j.
$$
Recall $y^{(t)}$ is in none of the $P_i^t$, so projecting onto $H$ and letting $t\to\infty$, we see that $v$ is not covered by any of the planks $P_i$. Thus the centered-limit construction produces the desired sign choice for the original family.

\subsubsection{Bang's generalization of Tarski's problem} \label{sec gen Bang}
Bang \cite{bang1951} formulated an affine invariant version of Tarski's plank problem (in the sense that it is invariant under affine transformations of the convex body). 
If $H$ is parallel to the boundary hyperplanes of a plank $P$, let us call the \textit{relative width} of $P$ the ratio 
$
w(P)/w(C,H).
$

\begin{conjecture}[Bang's conjecture] \label{conj Bang}
If a convex body is covered by a finite number of planks, then the sum of their relative widths is at least $1$.
\end{conjecture}

Note that this is strictly stronger than Tarski's plank problem since $w(C,H)\geq w(C)$. Here are a few more remarks regarding \cref{conj Bang}. First, using approximations and affine transformations,  Ohmann \cite{ohmann} showed that instead of working with an arbitrary number of planks, it suffices to work with $d$ planks in $\R^d$ that are mutually orthogonal and of width $1$, for all $d\in \N$.

Second, Hunter \cite{bangeq} looked for simple cases of equality in Bang's conjecture in the plane  when the number of planks is given but the convex body can vary. For instance, given a covering of a triangle by three planks, consider the equilateral triangle with sides of unit length that was obtained via affine transformations from the initial triangle. Then it is shown that the sum of the relative widths of the three planks is $1$ when and only when the sum of the segments $t_1+t_2+t_3$ in \cref{fig:tri} equals $1$, where the biggest triangle is the equilateral one prescribed above. The arguments of Hunter also show that \cref{conj Bang} holds for a covering by two planks in the plane, as was also proved in \cite{Alexander, bang54, gardner, moser}. However, note that the case of a covering by three planks is still open, and it was only recently \cite{sympl} that Bang's \cref{conj Bang} was proved for two directions of planks, that is, in the case when the planks can be partitioned into two parallel subfamilies.
\begin{figure}[ht] 
\centering
\begin{tikzpicture}[scale=0.85]

\draw [thick] (0,0)-- (6,0);
\draw [thick] (6,0)-- (3,5.196152422706633);
\draw [thick] (3,5.196152422706633)-- (0,0);
\draw [thick] (1.1733459118601275,2.032294734194974)-- (4.466846467004542,2.6554998149518623);
\draw [thick] (4.466846467004542,2.6554998149518623)-- (4,0);
\draw [thick] (4,0)-- (1.1733459118601275,2.032294734194974);
\draw [thick] (1.56,0)-- (5.488948822334847,0.885166604983954);
\draw [thick] (5.488948822334847,0.885166604983954)-- (1.9882239485470792,3.4437048957087493);
\draw [thick] (1.9882239485470792,3.4437048957087493)-- (1.56,0);
\draw (4,4) node[] {$s_2$};
\draw (5.1,2) node[] {$r_2$};
\draw (6,0.5) node[] {$t_2$};
\draw (0.35,1.2) node[] {$s_3$};
\draw (1.3,2.8) node[] {$r_3$};
\draw (2.2,4.4) node[] {$t_3$};
\draw (0.8,-0.23) node[] {$t_1$};
\draw (2.9,-0.23) node[] {$r_1$};
\draw (5,-0.23) node[] {$s_1$};


  
  
\end{tikzpicture}
\caption{Labeling of the segments appearing in Hunter's construction.} \label{fig:tri}
\end{figure}


Third,
Gardner \cite{gardner} connects Bang's conjecture with \textit{relative width measures} for a set of directions, which are Borel probability measures such that the measure of the intersection of $C$ with each plank orthogonal to one of the directions is equal to the relative width of that plank. He noted that if the convex body $C$ admits a relative width measure for all directions (in fact, even just for the coordinate directions in $\R^d$, for all $d\in\N$, as in Ohmann's observation), then Bang's conjecture holds for $C$. He then showed that measures with this property for the coordinate directions always exist for convex bodies in $\R^2$, yet they generally do not in $\R^3$.

Fourth,
 Alexander \cite{Alexander} observed that this conjecture is equivalent to the following generalization of a theorem of Davenport on intersections of straight lines with unit squares, which is ultimately related to Diophantine approximation in number theory (see \cref{section approx} for more details).

\begin{conjecture}[Reformulation of \cref{conj Bang}] \label{refor} If $C\subset \R^d$ is a convex body in Euclidean space sliced by $n-1$ hyperplane cuts, then one of the $n$ resulting pieces covers a translation of $\frac{1}{n}C$.
\end{conjecture}

This conjecture is known to hold when only successive cuts are used (cutting only one piece into two at each step) \cite{bezdekcon2} and is related to Conway's fried potato problem, which roughly asks what the radius of the largest ball you can inscribe in at least one of the resulting pieces is. As the authors of \cite{unsol} put it: ``In order to fry it as expeditiously as possible,
Conway wishes to slice a given convex potato into $n$ pieces by $n-1$ successive plane cuts
(just one piece being divided by each cut) so as to minimize the greatest inradius of the
pieces.''
This problem was solved by A. and K. Bezdek \cite{bezdekcon1, bezdekcon2} and their proof relies on Bang's \cref{Thm Bang}.

\cref{conj Bang} is a small modification to Tarski's conjecture, so one could expect that a similar solution exists. Yet, to this day it remains an important open problem in convex geometry. We shall still see that is has been proven in some specific cases which we will focus on in the upcoming sections.


\section{A plank theorem of Ball} \label{sec:3}

\subsection{Bang's conjecture for centrally symmetric convex bodies} 
In the 1990s, Ball proved Bang's conjecture in the case where $C$ is \textit{centrally symmetric} (for instance, $C$ symmetric with respect to the origin means that for all $x\in C$, we also have $-x\in C$). In the rest of this section, we shall refer to $C$ as simply being a symmetric convex body. 
This supposition on $C$ allows us to use tools from functional analysis. Indeed, any symmetric
convex body can be chosen as the unit ball defining a finite-dimensional normed space, and any such unit ball is a symmetric convex body.
To have a definition of a plank in a normed space that is analogous to that of a plank in Euclidean space, we define a plank in a normed space $X$ by
\begin{equation} \label{planche}
   \left\{x\in X : \abs{\phi(x)-m}\leq w\right\}
\end{equation}
for $\phi$ a continuous linear functional, $m$ a real number and $w$ a nonnegative real number. 
In what follows, we consider unit functionals, so that the relative width of a plank is exactly $w$. Here is Ball's result in this setting \cite{ball1991}.

\begin{thm} \label{thmboth}
If the unit ball of a Banach space $X$ is covered by a set of planks in $X$, then the sum of the relative widths of these planks is at least $1$.
\end{thm}

\subsection{Finite-dimensional Banach spaces}
 The proof ideas presented here mostly follow \cite{ball1991} and \cite{ballconvex}.
 We first consider the case where $X$ is finite-dimensional, and we shall treat the infinite-dimensional case in \cref{inf}. We still suppose that the covering consists of a finite number of planks. We can then reformulate \cref{thmboth} using these hypotheses and \eqref{planche}.

\begin{thm}[Plank theorem - Finite-dimensional case] \label{thm Ball}
Let $X$ be a finite-dimensional normed space. Let $(\phi_i)_{i=1}^n$ be a sequence of unit linear functionals in $X^*$, $(m_i)_{i=1}^n$ real numbers and $(w_i)_{i=1}^n$ nonnegative real numbers such that $\sum_{i=1}^n w_i =1$. Then there exists $x$ in the unit ball of $X$ such that
$$
\abs{\phi_i(x)-m_i}\geq w_i
$$
for all $1\leq i\leq n$.
\end{thm}

If \cref{thmboth} fails, it means there is a family of closed planks covering the unit ball whose total width is $<1$. By replacing them with slightly wider open planks such that the total width is 
exactly $1$,
we can apply \cref{thm Ball}. Hence, the previous theorem can be thought of as a quantitative version of the contrapositive of \cref{thmboth} in the finite-dimensional case. It further justifies that we may choose to work with closed or open planks, as alluded to in the comment following the proof of \cref{lemBangV2}.

\subsubsection{Reduction of the plank theorem to a linear algebra problem} \label{sec:reduc}
Ball's solution uses a clever mix of tools from functional analysis and linear algebra. The main idea is to obtain a result analogous to Bang's lemma that holds for matrices that are not necessarily symmetric. We first reduce the problem to purely combinatorial propositions by mimicking some of the ideas that allowed us to express Bang's lemma in terms of Gram matrices.

Let $\phi_j$ be linear functionals as in the statement of \cref{thm Ball}. Since $X$ is finite-dimensional, the unit ball of $X$ is compact and, in particular, the norm of each $\phi_j$ is attained: we can find $x_j$ in the unit ball of $X$ such that $\phi_j(x_j)=1$. Given such $x_j$, define the $n {\times} n$ real matrix $A$ by $a_{ij}=\phi_i(x_j)$. If $(\lambda_j)_{j=1}^n$ is a sequence of real numbers such that $\sum_{j=1}^n \abs{\lambda_j}\leq 1$, then $x:=\sum_{j=1}^n \lambda_j x_j$ is such that
$\norm{x}\leq 1$ and $\phi_i(x)=\sum_{j=1}^n a_{ij}\lambda_j$ by linearity. In short, it will suffice to prove the following theorem.

\begin{thm} \label{thm Ball V2}
Let $A=(a_{ij})_{1\leq i,j\leq n}$ be an $n {\times} n$ real matrix with $1's$ on the diagonal, $(m_i)_{i=1}^n$ real numbers and $(w_i)_{i=1}^n$ nonnegative real numbers such that $\sum_{i=1}^n w_i= 1$. Then there exist real numbers $(\lambda_j)_{j=1}^n$ such that $\sum_{j=1}^n \abs{\lambda_j}\leq 1$ and 
$$\Bigl|\sum_{j=1}^n a_{ij}\lambda_j - m_i\Bigr| \geq w_i$$ for all $1\leq i\leq n$.
\end{thm}

Recall that we interpret $w_i$ as the width of the $i$th plank. For arbitrary planks whose total width is $<1$, we may increase the width of each plank so that it becomes rational, yet the total width remains $<1$. Then changing the planks to have rational widths with a common denominator $N$, one can split them further into planks of width $1/N$. By this density argument, we may suppose that $w_i=1/n$ for all $i$ in what follows. (Here we are implicitly using the fact that the statement of the theorem fails if and only if the convex body is covered by the interiors of a finite number of planks with total width equal to $1$.)
Also note that if we can show $\sum_{j=1}^n \lambda_j^2\leq 1/n$, then 
$$
\Bigl(\sum_{j=1}^n \abs{\lambda_j}\Bigr)^2 \leq n\sum_{j=1}^n \lambda_j^2 \leq 1
$$
by the Cauchy--Schwarz inequality.

\subsubsection{Proof of \cref{thm Ball V2}}
The proof of \cref{thm Ball V2} will rely on Bang's lemma and two new lemmas.

\begin{lem} \label{cor final}
Let $H$ be a positive semidefinite symmetric matrix with $1's$ on its diagonal. If $U$ is orthogonal, then
$$
\sum_{i=1}^n \left(HU\right)^2_{ii} \leq n.
$$
\end{lem}

We need to introduce an orthogonal matrix to use \cref{cor final}.

\begin{lem} \label{lem Ball 2}
Let $A$ be an $n{\times}n$ real matrix with nonzero rows. Then there exist positive real numbers $(w_i)_{i=1}^n$ and an orthogonal matrix $U$ such that the matrix $H$ defined by
$$h_{ij}:=w_i (AU)_{ij}$$ is positive semidefinite, symmetric, and only has $1's$ on its diagonal.
\end{lem}

We omit the proofs of these two propositions because they do not rely on  ideas that will be used again in the rest of this survey. \cref{cor final} is a direct consequence of the Cauchy--Schwarz inequality stated in terms of the trace of a matrix, and \cref{lem Ball 2} is technical and mostly follows from the polar decomposition of matrices (see \cite{ball1991} for detailed proofs).

By \cref{lem Ball 2}, we may take positive real numbers $(w_j)_{j=1}^n$ and an orthogonal matrix $U$ such that $H:=(w_i(AU)_{ij})_{1\leq i,j\leq n}$ is symmetric, positive semidefinite, and has $1's$ on the diagonal. It then follows from Bang's lemma, in the version formulated in the paragraph following \cref{lemBangV2}, that we can find $(\epsilon_j)_{j=1}^n$
such that, for all $i$,
$$
w_i\leq \Bigl|\sum_{j=1}^n h_{ij}\epsilon_jw_j - nw_i m_i\Bigr| =\Bigl|w_i\sum_{j=1}^n (AU)_{ij}\epsilon_jw_j - nw_i m_i\Bigr|. 
$$
Dividing by $n\cdot w_i$, this is equivalent to 
$$
\frac{1}{n}\leq \Bigl|\sum_{k=1}^n a_{ik}\Bigl(\frac{1}{n}\sum_{j=1}^nu_{kj}\epsilon_jw_j\Bigr) - m_i\Bigr|,
$$
which motivates the definition
$$\lambda_k := \frac{1}{n}\sum_{j=1}^nu_{kj}\epsilon_jw_j.$$
Since $U$ is orthogonal and $\epsilon_j \in \{\pm 1\}$, we obtain
$$
\sum_{k=1}^n \lambda_k^2 =\frac{1}{n^2}\sum_{j=1}^n w_j^2.
$$
Also, since  $H=(w_i(AU)_{ij})$, we see that
$$
w_ia_{ij}=(HU^*)_{ij}.
$$
Taking $i=j$ and using $a_{ii}=1$ yields $w_i=(HU^*)_{ii}$. Recall that we want
$$
\sum_{k=1}^n \lambda_k^2 \leq \frac{1}{n},
$$
so we must show
$$
\sum_{j=1}^n (HU^*)^2_{jj} \leq n,
$$
which is precisely the statement of \cref{cor final} because $U^*$ is orthogonal. 

\subsection{Infinite-dimensional Banach spaces} \label{inf}
The goal of this subsection is to prove \cref{thmboth} in the case where $X$ is an infinite-dimensional Banach space. For reflexive spaces, the result follows readily from the finite-dimensional case (indeed, reflexivity implies that the norm of each functional is
attained at some unit vector, as in the beginning of \cref{sec:reduc}). In the general case, the modifications needed in the proof 
are not obvious, but all the tools we shall need are already available to us. However, since the proof is similar to the finite-dimensional case, we only highlight the differences and refer to \cite{ball1991} for the details. Recall that 
in the finite case it suffices to obtain a sequence of real numbers $(\lambda_j)_{j=1}^n$ that sum to at most $1$. However, we obtain something stronger, namely that $\sum_{j=1}^n\lambda_j^2\leq 1/n$ if we suppose that $w_i=1/n$ for all $i$ for simplicity. Running the proof of \cref{thm Ball V2} again with an arbitrary choice of $w_i$ would instead yield $\sum_{j=1}^n w_j^{-1}\lambda_j^2 \leq \sum_{j=1}^n w_j$. We can thus use the following lemma.

\begin{lem} \label{lem infini}
Let $A=(a_{ij})_{1\leq i,j\leq n}$ be an $n {\times} n$ real matrix with $1's$ on the diagonal, $(m_i)_{i=1}^n$ real numbers and $(w_i)_{i=1}^n$ nonnegative real numbers.
Then there exist $(\lambda_j)_{j=1}^n$ such that $\sum_{j=1}^n w_j^{-1}\lambda_j^2 \leq \sum_{j=1}^n w_j$ and $$\Bigl|\sum_{j=1}^n a_{ij}\lambda_j - m_i\Bigr| \geq w_i$$ for all $1\leq i\leq n$.
\end{lem}

Also note that it suffices to prove an ``infinite'' version of \cref{thm Ball}, where we replace every sequence of $n$ elements with a sequence indexed by $\N$.

\begin{thm}[Plank theorem - Infinite-dimensional case] \label{thm Ball infini}
    Let $X$ be a Banach space. Let $(\phi_i)_{i=1}^\infty$ be a sequence of unit functionals in $X^*$, $(m_i)_{i=1}^\infty$ real numbers and $(w_i)_{i=1}^\infty$ nonnegative real numbers such that $\sum_{i=1}^\infty w_i < 1$. Then there exists $x$ in the unit ball of $X$ such that
$$
\abs{\phi_i(x)-m_i} > w_i
$$
for all $i$.
\end{thm}

The idea is to
choose $v_i$ such that $v_i>w_i$ for all $i$, but also such that their sum remains smaller than $1$, say $\sum_{i=1}^\infty v_i=1-\varepsilon<1$. We can find $x_i$ in the unit ball of $X$ such that $\phi_i(x_i)=1-\varepsilon$, then we apply \cref{lem infini} with a matrix whose $i,j$ entry is $\phi_i(x_j)/(1-\varepsilon)$. 
For every $n$, we extend the sequences $(\lambda_j^{(n)})_{j=1}^n$ that we obtained from \cref{lem infini} to infinite sequences by completing them with zeros. Using a diagonal argument, we obtain a subsequence of $(\lambda^{(n)}_j)_{n=1}^\infty$ that converges pointwise for every $j$. This allows us to construct the desired point $x$.

Note that in the finite case, we used compactness to find points in the unit ball of $X$ where the norm of the functionals was attained, but now we can only get away with almost attaining their norm.

\section{Other types of plank problems and polynomial analogues} \label{sec:app1}
We refer to the Theorems \ref{thm Ball} and \ref{thm Ball infini} of Ball as plank theorems without any further distinction, yet they stem from a symmetric plank problem on Banach spaces, and
as we shall see in the remaining sections, there are many variations and generalizations of these problems.

\subsection{The complex plank problem} 
\cref{thm Ball infini} is sharp as can be seen by taking $X=\ell_1$ and $\phi_i$ to be the standard basis vectors of $\ell_\infty$, the dual of $\ell_1$. However, it is expected that we can improve this result for Hilbert spaces. For instance, Ball proved the following complex Hilbert space analogue of \cref{thm Ball} in \cite{ball2001}, which turned out to be essential in solving many other problems, as will be discussed in subsequent sections. 

\begin{thm}[Complex plank theorem] \label{Ball complexe}
Let $(v_j)_{j=1}^n$ be unit vectors in a complex Hilbert space, and $(t_j)_{j=1}^n$ a sequence of nonnegative real numbers such that $\sum_{j=1}^n t_j^2=1$. Then, there exists a unit vector $z$ such that $$\abs{\inner{v_j}{z}}\geq t_j$$ for all $1\leq j\leq n$.
\end{thm}


We remark that this plank problem does not involve shifts: it is a complex analogue of the symmetric plank problem for centered planks \textit{only} (see \cref{sec:revisit} for a recent, more general complex plank theorem which applies to planks that are not necessarily centrally symmetric). Also,
Ball's proof of the previous theorem has a lot of overlap with the demonstration of \cref{thm Ball V2}, so we omit the demonstration here. Let us just mention that the proof reduces to a linear algebra statement as in \cref{sec:reduc}. In particular, this is a reformulation of the case of equal widths $t_j=1/\sqrt{n}$, and the general case follows from technical considerations as outlined in \cite{ball2001}.

\begin{proposition} \label{prop:comp}
Let $H$ be an $n\times{n}$ complex positive semidefinite hermitian matrix with $1's$ on its diagonal. Then there exist complex numbers $(w_j)_{j=1}^n$ such that $\abs{w_j}\leq 1$ and
\begin{equation} \label{eq thm}
w_j\sum_k h_{jk}\overline{w_k}=1
\end{equation}
for all $1\leq j\leq n$.
\end{proposition}

To prove this proposition, we can interpret \eqref{eq thm} as a system of $2n$ equations by splitting the real and imaginary parts and we can show the existence of a solution using Lagrange multipliers. We then form a new matrix with diagonal entries equal to $\abs{w_j}^2$, and we show that all of these are $\leq 1$ using the maximum principle in complex analysis, which finishes the demonstration (see \cite{ball2001} for the details).

\subsubsection{Streamlined proof of the complex plank problem} \label{sec:stream1}
It is important to note that recently, a new and simplified proof of \cref{Ball complexe} was given by Ortega-Moreno \cite{ortegarevis2021} using a versatile polynomial approach, which has proven useful in more recent work on elementary questions regarding zeros of polynomials restricted to the unit sphere. Let us explain his approach by outlining a proof of the complex plank problem for equal widths. The more general case follows from similar ideas.

\begin{theorem} \label{thm:om}
Let $v_1,v_2,\ldots,v_n$ be a sequence of unit vectors in $\C^d$. Then there exists a unit vector $u\in \C^d$ such that
$$
|\inner{v_k}{u}|\geq 1/\sqrt{n}
$$
for all $1\leq k\leq n$.
\end{theorem}
The statement of \cref{thm:om} holds when $u$ maximizes $\prod_{i=1}^n|\inner{v_i}{u}|$ among unit vectors. To show this arguing by contradiction, suppose $|\inner{v_1}{u}|<1/\sqrt{n}$. Let $u_z:=zv_1+u$ whenever $z$ is such that $u_z\in S^d$. This happens if and only if $z$ is on the circle $C$ of radius $|\inner{v_1}{u}|$ centered at $-\inner{u}{v_1}$. Now let $p$ be the polynomial
$$
p(z)=\prod_{i=2}^n \frac{\inner{u_z}{v_k}}{\inner{u}{v_k}}.
$$
We can obtain a contradiction by considering the derivative $\partial/\partial\Bar{z}$ of $p$ at $z=0$ (where $|p(z)|$ reaches its maximum on $C$) and applying the maximum modulus principle. We omit the details here because this theorem, albeit simple and straightforward, is dominated by a polynomial analogue statement that we will encounter in \cref{sec:polynomialapproach} while explaining how this polynomial approach ties in with diverse plank problems.

\subsection{The spherical plank problem} \label{sec:zoneconj}
This section covers a long-standing conjecture that was recently proven to be true using arguments akin to Bang's lemma and the plank theorems.

A \textit{great circle} is the intersection of a sphere in $\R^{d+1}$ and a plane passing through its center, and the \textit{spherical distance} denotes the shortest distance between two points on the sphere along its surface. These concepts allow us to state a spherical version of the plank problems, where the other definitions are taken \textit{mutatis mutandis}. Indeed, given a great circle, we say that the analogue of a plank of width $w$ is a \textit{zone} of width $w$, that is, the set of points that are at spherical distance $w/2$ or less away from the great circle. Note that zones are not convex in spherical geometry.

\subsubsection{Fejes Tóth's zone conjecture} \label{sec:fejes}
In 1973, the mathematician L. Fejes Tóth \cite{monthly} conjectured that if $n$ zones of equal width $w$ cover the unit sphere, then $w$ is at least $\pi/n$. He was partially motivated by the optimal establishment of $n$ operating stations on a new planet  or the most economical way to explore a planet using $n$ satellites. He also formulated a more general conjecture where the widths of the zones may vary.

\begin{conjecture}[Fejes Tóth's zone conjecture] \label{conj Fejes}
The sum of the widths of the zones covering the unit sphere is at least $\pi$.
\end{conjecture}

\begin{figure}[ht]
    \centering
    \begin{tikzpicture}[scale=3]
    \draw [thick] (0,0) circle (1cm);
\draw [shift={(0,0)},line width=2.8pt,color=ffvvqq]  plot[domain=0.6265659771973224:1.254845386734214,variable=\t]({1*1*cos(\t r)+0*1*sin(\t r)},{0*1*cos(\t r)+1*1*sin(\t r)});
\draw [shift={(0,0)},line width=2.8pt,color=ccqqqq]  plot[domain=1.8808082293744517:2.513021172702312,variable=\t]({1*1*cos(\t r)+0*1*sin(\t r)},{0*1*cos(\t r)+1*1*sin(\t r)});
\draw [shift={(0,0)},line width=2.8pt,color=zzttff]  plot[domain=3.141592653589793:3.7681586307871155,variable=\t]({1*1*cos(\t r)+0*1*sin(\t r)},{0*1*cos(\t r)+1*1*sin(\t r)});
\draw [shift={(0,0)},line width=2.8pt,color=qqwuqq]  plot[domain=4.396438040324007:5.022400882964245,variable=\t]({1*1*cos(\t r)+0*1*sin(\t r)},{0*1*cos(\t r)+1*1*sin(\t r)});
\draw [shift={(0,0)},line width=2.8pt,color=qqqqff]  plot[domain=-0.6285714808874809:0,variable=\t]({1*1*cos(\t r)+0*1*sin(\t r)},{0*1*cos(\t r)+1*1*sin(\t r)});
\draw [ultra thick,dashed,color=ffvvqq] (0.3107204813761027,0.9505013321681369)-- (-0.8100458765360872,-0.5863665047620662);
\draw [ultra thick,dashed,color=ffvvqq] (0.8100458765360872,0.5863665047620662)-- (-0.3107204813761027,-0.9505013321681369);
\draw [ultra thick,dashed,color=ccqqqq] (-0.3050699716479127,0.9523299388335651)-- (0.808868288115337,-0.5879898744718011);
\draw [ultra thick,dashed,color=ccqqqq] (-0.808868288115337,0.5879898744718011)-- (0.3050699716479127,-0.9523299388335651);
\draw [ultra thick,dashed,color=qqqqff] (1,0)-- (-0.808868288115337,0.5879898744718011);
\draw [ultra thick,dashed,color=qqqqff] (0.808868288115337,-0.5879898744718011)-- (-1,0);
\draw [ultra thick,dashed,color=zzttff] (-1,0)-- (0.8100458765360872,0.5863665047620662);
\draw [ultra thick,dashed,color=zzttff] (-0.8100458765360872,-0.5863665047620662)-- (1,0);
\draw [ultra thick,dashed,color=qqwuqq] (-0.3107204813761027,-0.9505013321681369)-- (-0.3050699716479127,0.9523299388335651);
\draw [ultra thick,dashed,color=qqwuqq] (0.3050699716479127,-0.9523299388335651)-- (0.3107204813761027,0.9505013321681369);
\draw (0.65,0.9) node[] {$z_1$};
\draw (-0.65,0.9) node[] {$z_2$};
\draw (-1.1,-0.3) node[] {$z_3$};
\draw (0,-1.1) node[] {$z_4$};
\draw (1.08,-0.3) node[] {$z_5$};
\end{tikzpicture}
\caption{Five zones of equal width covering the unit sphere $S^2$}
\end{figure}

It is straightforward to generalize the previous definitions and \cref{conj Fejes} to the unit spheres $S^d$ embedded in $\R^{d+1}$, but note that the conclusion of the conjecture would stay the same: it is independent of the ambient dimension. 
Also, papers by Rosta \cite{rosta} and Linhart \cite{linhart} deal with the special cases of coverings by $3$ or $4$ zones, respectively.
It was only in 2017 that Jiang and Polyanskii \cite{LFT} proved \cref{conj Fejes} for the more general unit $d$-spheres. Their proof is partly inspired by a paper of Goodman and Goodman \cite{goodman} on covering of a nonseparable collection of similar copies of a convex body, where the main result is the following circle covering theorem conjectured by Erd\H{o}s.

\begin{theorem}
If $n$ circles with respective radius $r_i$ lie in a plane such that no line divides the circles into two nonempty sets without intersecting any circle, then the $n$ circles can be covered by one circle of radius $\sum_{i=1}^nr_i$.
\end{theorem}

As noted in \cite{LFT}, the spherical analogue of this result should be about \textit{spherical caps}, the set of points within fixed spherical distance of a given point on the unit $d$-sphere, which is the dual of a zone via projective duality (which interchanges a line through the origin with its orthogonal
hyperplane through the origin). The cap covering theorem of Polyanskii \cite{capcover} is both a spherical analogue of the circle covering theorem of Goodman and Goodman and a strengthening of the zone conjecture. We remark that its proof implicitly uses the concept of multi-planks introduced by Balitskiy \cite{multi}.

Since Jiang and Polyanskii's resolution of the zone conjecture, streamlined alternative proofs were found and we will present them in \cref{sec:polynomialapproach}, so let us just briefly describe Jiang and Polyanskii's approach to highlight the similarities it shares with Bang's solution to Tarski's plank problem. 
Suppose that $n$ zones of width $2\alpha_i$ cover $S^d$ embedded in $\R^{d+1}$, but their total width is less than $\pi$. Let $H_i$ be the middle hyperplane of the $i$th zone and $u_i$ be a unit normal vector of $H_i$. Consider the set of points of the form 
\begin{equation*}
    \Bigl\{\sum_{j=1}^n \epsilon_ju_j\sin\alpha_j : \epsilon_j\in \{\pm 1\}\Bigr\},
\end{equation*} 
and let $w$ be a point in this set having maximal norm. Then the fact that $\|w\|\geq \|w-2\epsilon_iu_i\sin\alpha_i\|$ for all $i$ implies that the distance from $w$ to $H_i$ is at least $\sin\alpha_i$. If $\|w\|<1$, then $w/\|w\|$ is an element of $S^d$ that is not covered by the $P_i$. 
We remark that the preceding argument used to find a point that is not covered by the convex hulls of the zones is reminiscent of Bogn\'ar's trick as described in \cref{sec:dig}. It would also be possible to find it by maximizing a well-chosen quadratic function on $\{\pm 1\}^n$, as mentioned in \cite{LFT}, which this time is reminiscent of Bang's lemma. 
In the case where $\|w\|\geq 1$, it can be shown that several zones can be replaced by one zone without increasing the total width, and this argument can be repeated (see \cite{LFT} for details). We remark that similar ideas are used in \cite{capcover} to prove a stronger version of the zone conjecture.


\subsubsection{A few remarks on Fejes Tóth's zone conjecture} \label{sec:reforzone}
One may wonder what an optimal plank theorem in real Hilbert spaces would look like. It turns out that a statement as strong as \cref{Ball complexe} cannot hold,
and a rather simple counterexample has been observed by a few people: consider $n$ equally spaced points on the half-circle, and with their $n$ antipodal points, call the resulting vectors $v_1,\ldots,v_{2n}$. 
Then, for all unit vectors $u$, we can find a $k$ such that $\abs{\inner{v_k}{u}}\leq \sin(\pi/2n)$, which can be made arbitrarily small.

Still, by giving an explicit description of a zone, we can formulate a spherical plank problem that resembles the previous plank theorems in the case of centered planks. If $v$ is a unit vector on a given great circle, the associated zone of width $w$ is   
$$
\{x\in S^1 : \abs{\inner{v}{x}} \leq \sin(w/2)\}.
$$
Then, the following is a reformulation of \cref{conj Fejes} in the case of zones of equal width and can be thought of as an optimal plank problem for real Hilbert spaces, in view of the previous example.



\begin{conjecture}[Reformulation of Fejes Tóth's zone conjecture for zones of equal width] \label{fejes ref}
Let $v_1,v_2,\ldots,v_n$ be a sequence of unit vectors in a real Hilbert space $\mathcal{H}$. Then there exists a unit vector $u\in \mathcal{H}$ such that
$$
\abs{\inner{v_k}{u}}\geq \sin(\pi/2n)
$$
for all $1\leq k\leq n$.
\end{conjecture}

An alert reader might have noticed that this is nothing but a spherical variant of \cref{thm:om}.
We remark that in 2021, a new proof of the Fejes Tóth zone conjecture (for zones of equal width) was given by Ortega-Moreno \cite{inverse}, starting with the previous reformulation. The approach he employs is quite different from the one in Jiang and Polyanskii's paper and uses the somewhat new concept of inverse eigenvectors in convex geometry, which are used to transform a geometry problem into a study of the behavior of polynomials (they were complex in the complex plank problem, and they are trigonometric in the proof of \cref{fejes ref}). We say that a vector $x$ is an \textit{inverse eigenvector of $M$} if $Mx=x^{-1}$, where by definition the $i$th component of $x^{-1}$ is the inverse of the $i$th component of $x$. These inverse eigenvectors were only formally introduced and defined in \cite{ambrus2019analytic}, but they appear implicitly in the proof of Ball's plank theorem in the complex case, and in fact you can interpret \cref{prop:comp} as saying that every complex positive semidefinite hermitian matrix with $1's$ on the diagonal has an inverse eigenvector in the complex $\ell_\infty$ unit ball.

Zhao \cite{zhao} recently simplified the proof of Ortega-Moreno by eliminating the need for the notion of inverse eigenvectors, using instead the polynomial approach that we alluded to in \cref{sec:stream1}. In fact, this is a striking example of how bringing in tools from trigonometry and elementary analysis greatly simplifies a problem that seemed intractable for decades.
Let us outline his short and neat argument.

We will show that the statement of \cref{fejes ref} holds when $u$ is a unit vector that maximizes $\prod_{i=1}^n\abs{\inner{v_i}{u}}$. For the sake of contradiction, suppose that $\abs{\inner{v_1}{u}}<\sin(\pi/2n)$ and consider the plane spanned by $\{u,v_1\}$. Then we can take a vector $w\perp u$ with $|w|<1$ such that, if 
$$u_{\theta}:=(\cos\theta) u+(\sin\theta)w,$$ 
then $u_{\pi/2n}\perp v_1$ (\cref{fig:yufei}). 

\begin{figure}[ht]
    \centering
    \begin{tikzpicture}[scale=1]
\draw [rotate around={0:(0,0)},thick] (0,0) ellipse (2.8284271247461907cm and 2cm);
\draw [thick,-to] (0,0)-- (2.8284271247461907,0);

\draw [thick, -to] (0,0)-- (0,2);

\draw [thick, -to] (0,0)-- (1.1158732233165483,-2.282644797340673);

\draw [thick, -to] (0,0)-- (2.3265641426375954,1.137343239790213);
\draw [thick] (0,0)-- (0.06681396845940184,-0.13667552398134122);
\draw [thick] (0.06681396845940184,-0.13667552398134122)-- (0.20348949244074305,-0.06986155552193941);
\draw [thick] (0.20348949244074305,-0.06986155552193941)-- (0.13667552398134125,0.06681396845940182);
\draw [thick] (0.13667552398134125,0.06681396845940182)-- (0,0);
\draw [thick,dashed] (0,0)-- (0,-2);
\draw [shift={(0,0)},thick]  plot[domain=4.71238898038469:5.167077685971319,variable=\t]({1*0.5900115966663766*cos(\t r)+0*0.5900115966663766*sin(\t r)},{0*0.5900115966663766*cos(\t r)+1*0.5900115966663766*sin(\t r)});
\draw (-0.5,2.4315696157953175) node[anchor=north west] {$w=u_{\pi/2}$};
\draw (2.4108258675813103,1.4726877098815474) node[anchor=north west] {$u_{\pi/2n}$};
\draw (2.9298536882135364,0.2) node[anchor=north west] {$u=u_0$};
\draw (1.1264519385591911,-1.9053916650807246) node[anchor=north west] {$v_1$};
\draw (-0.95,-0.04041034348697021) node[anchor=north west] {$<\frac{\pi}{2n}$};
\end{tikzpicture}
\caption{Vectors appearing in the proof of \cref{fejes ref}} \label{fig:yufei}
\end{figure}

Now consider the function 
$$f(\theta)=\prod_{i=1}^n \frac{\inner{v_i}{u_\theta}}{\inner{v_i}{u}}.$$
We can obtain a contradiction by bounding the number of zeros of $f(\theta)-\cos(n\theta)$ in $[0,2\pi)$.
Namely, it can be proven using the intermediate value theorem that $f(\theta)-\cos(n\theta)$ has at least $2n$ distinct zeros in $[0,2\pi)$, and from considerations on the derivative of $f$, that $f(\theta)-\cos(n\theta)$ has at most $2n-2$ distinct zeros in $[0,2\pi)$.
We omit the proofs here because they are eclipsed by those of a polynomial zone theorem that we will see in the next section.

\subsection{A polynomial approach to plank problems} \label{sec:polynomialapproach}
Here we unite the previous comments on the polynomial approach that was recently used to simplify the demonstration of some plank problems. If we restrict Bang's \cref{Thm Bang} to the case of the unit ball, then we know that the sum of the widths of the planks in a covering is at least $2$. In particular, if there are $n$ planks of equal width $w$, then $w\geq 2/n$, which can be interpreted as follows: \textit{for any collection of $n$ hyperplanes in Euclidean space, there exists a point in the unit ball that is at distance at least $1/n$ from the union of the hyperplanes}. Here we note that a hyperplane is the zero set of
a degree $1$ polynomial, and the union of $n$ hyperplanes is the zero set of a degree $n$ polynomial. It is then straightforward to see how the formulations of \cref{thm:om} and \cref{fejes ref} might come to be.  In 2022, Glazyrin, Karasev, and Polyanskii \cite{avoidingzeros} devised a polynomial approach to obtain general results on the zeros of polynomials restricted to the unit sphere, which bridged the gap between Tarski's plank problem and its complex and spherical variants. These ideas can be traced back to the PhD thesis of Ambrus \cite{ambrus2019analytic} and also to work of Ambrus, Ball, and Erd\'elyi \cite{2013} on the strong polarization problem (see the end of \cref{sec:polarization} for more details). 

\subsubsection{Revisiting the complex and spherical plank problems} \label{sec:revisit}
Zhao \cite{zhao} was the first to prove a plank theorem using these ideas
by giving a new proof of the zone conjecture for equal widths, and soon after, Ortega-Moreno \cite{ortegarevis2021} leveraged the method to simplify the proof of the complex plank problem (see also Sections \ref{sec:stream1} and \ref{sec:reforzone}). 
The following theorems from \cite{avoidingzeros} encapsulate the essence of the arguments of Zhao and Ortega-Moreno by means of results on avoiding zeros of polynomials.

\begin{theorem} \label{thm:polsph}
If a polynomial $P\in \R[x_1,\ldots,x_d]$ of degree $n$ has a nonzero restriction to $S^{d-1}\subset \R^d$ and attains its maximal absolute value on $S^{d-1}$ at a point $p$, then $p$ is at angular distance at least $\frac{\pi}{2n}$ from the intersection of the zero set of $P$ with $S^{d-1}$.
\end{theorem}

\begin{theorem} \label{thm:polcom}
If a homogeneous polynomial $P\in \C[z_1,\ldots,z_d]$ of degree $n$ is not identically zero and attains its maximal absolute value on $S^{2d-1}\subset \C^d$ at a point $p$, then $p$ is at angular distance at least $\arcsin\frac{1}{\sqrt{n}}$ from the intersection of the zero set of $P$ with $S^{2d-1}$.
\end{theorem}

The homogeneity hypothesis in Theorem \ref{thm:polcom} is essential (see Remark 1.12 in \cite{avoidingzeros}). The authors of \cite{complnou} found a nonhomogeneous extension of the complex polynomial plank problem which generalizes the complex plank theorem of Ball to planks that are not necessarily centrally symmetric. Since geometrically, a plank in the complex setting is closer to a circular cylinder, here given $\delta_i$ a positive real number, we refer to the $\delta_i$-neighborhood of a complex affine hyperplane in $\mathbb{C}^d$ as a $\delta_i$-\textit{cylinder}.

\begin{theorem}
If $\delta_1,\ldots,\delta_n>0$ satisfy $\sum_{i=1}^n\delta_i^2<R^2$, then the centered ball of radius $R$ in $\mathbb{C}^d$ cannot be covered by a union of $n$ $\delta_i$-cylinders.
\end{theorem}

We shall only prove \cref{thm:polsph} to showcase the elementary ideas behind these kinds of results, and we refer interested readers to \cite{avoidingzeros} and \cite{complnou} for the rest. We will also show a stronger result than the zone conjecture: if \textit{spherical segments} (the portion of a sphere between two parallel hyperplanes, which need not be centrally symmetric) cover $S^{d-1}$, then the sum of their widths is at least $\pi$. As noted by the authors, the techniques of Jiang and Polyanskii described in \cref{sec:fejes} could not be modified to prove this stronger statement.
Let us see
how the spherical plank problem reduces to this polynomial result. 
The idea is to approximate spherical segments by unions of thin spherical segments of the same width and apply the polynomial result to the polynomial
$$
H_1\cdots H_n,
$$
where $H_i=0$ is the linear equation of the middle hyperplane of the $i$th segment. The general case is only marginally harder, and in fact the argument is similar to the one that was used to justify picking the same rational width for each plank in \cref{sec:reduc}. Indeed, we obtain a point $x\in S^{d-1}$ at spherical distance at least $\frac{\pi}{2n}$ from every middle hyperplane, and we then proceed by contradiction. First suppose the covering consists of $n$ spherical segments of equal width $2\delta$. In particular, $2\delta n<\pi$, or equivalently, $\frac{\pi}{2n}>\delta$, so that $x$ is outside every segment. Second, suppose the $n$ spherical segments have rational width with common denominator $N$. Splitting every spherical segment into several segments of equal width $1/N$, this reduces to the previous case. Third, in the general case, suppose once more that the total width is $<\pi$. By density of the rationals in the reals, we may increase the width of each segment so that it becomes rational, yet the total width remains $<\pi$, and the result follows from the previous cases. 
Now let us turn to the proof of \cref{thm:polsph} \cite{avoidingzeros}.


\begin{proof}[Proof of \cref{thm:polsph}]
Suppose $x_0$ is a zero of $P$ whose angular distance with $p$ is minimal. By considering the linear span of $x_0$ and $p$, we see that it suffices to prove the result for $d=2$.
We work with the restriction of $P$ to the unit circle, which allows us to consider $P$ as a trigonometric polynomial of degree $n$ by 
switching to polar coordinates. We can assume $P$ has maximal absolute value $M$ at the origin (if not, we simply shift $P$). Then our goal is to show that the distance between $0$ and the zeros of $P$ is $\geq \frac{\pi}{2n}$. Consider the trigonometric polynomial
$$
Q(x)= P(x)\pm M\cos(nx).
$$
We can assume that $Q$ has zeros of multiplicity $2$ at $2\pi k$, $k\in\Z$, by choosing the sign $\pm$ appropriately. We will work over a $2\pi$-period around the origin, but by this assumption, the argument would work more generally.
For $m=1,\ldots,2n-1$,
$$
Q\Bigl(\frac{m\pi}{n}\Bigr) = P\Bigl(\frac{m\pi}{n}\Bigr) \pm (-1)^m M
$$
has sign $\pm (-1)^m$ or $Q$ has a zero of multiplicity $2$ at $\pi m/n$. The intermediate value theorem then yields that there are at least $2n-2$ zeros of $Q$ on $(\pi/n,(2n-1)\pi/n)$, counted with multiplicity.
Since $Q$ also has a zero of multiplicity $2$ at the origin, $Q$ has at least $2n$ zeros per $2\pi$-period. If $Q$ has more zeros (e.g., at $-\pi/n$ or $\pi/n$), then $Q$ is identically zero and the conclusion holds; otherwise, $Q$ has no zeros on $[-\pi/n,\pi/n]$, except at $0$.
In this case, it must be strictly positive (respectively strictly negative) on $(0,\pi/n)$ (respectively $(-\pi/n,0)$); but if $P(x_0)=0$ for some $x_0$ in $(-\frac{\pi}{2n},\frac{\pi}{2n})$, then $\cos(nx_0)>0$ and so $Q(x_0)$ is of opposite sign to what was derived above, a contradiction.
Hence there are no zeros of $P$ in $(-\frac{\pi}{2n},\frac{\pi}{2n})$, which yields the desired result. Note that in the above, the bound is sharp precisely when $Q=0$.
\end{proof}

\subsubsection{More remarks on the polynomial analogues of plank problems}
%
It is worth noting that this polynomial approach works on the complex and spherical plank problems because they do not involve shifts, they concern centered planks. One could say that they are really polarization problems (as in \cref{sec:polarization}).
It follows that these polynomial arguments are unlikely to yield a streamlined demonstration of Ball's \cref{thm Ball}. Finding a simplification of the proof of the original plank theorem of Ball is thus an intriguing problem. Nonetheless, a tight polynomial analogue of Bang's \cref{Thm Bang} for a ball can be obtained with a little more work using these ideas \cite{avoidingzeros}.

\begin{theorem} \label{thm:tightball}
For every $n\geq 1$, there exists a sequence of analytic functions $G_n:\R\to\R$ such that for every nonzero polynomial $P\in \R[x_1,\ldots,x_d]$ of degree $n$, the point in the unit ball $B^d$ of $\R^d$ where $P(x)G_n(|x|)$ attains its maximum on $B^d$ is at distance at least $1/n$ from the zero set of $P$.
\end{theorem}
The functions $G_n$ are explicitly given by $g_n(n\pi x/2)$, where
$$
g_n(x):=
\begin{dcases*}
  \prod_{i=\frac{n}{2}+1}^\infty \Bigl(1-\Bigl(\frac{2x}{(2i-1)\pi}\Bigr)^2\Bigr)  & if $n$ is even, \\
  \prod_{i=\frac{n+1}{2}}^\infty \Bigl(1-\Bigl(\frac{x}{i\pi}\Bigr)^2\Bigr) & if $n$ is odd,
\end{dcases*}
$$
which is to be compared with the infinite product representation of the trigonometric functions $\sin(x)$ and $\cos(x)$.

The proof of \cref{thm:tightball} is more intricate than the one from \cref{sec:revisit}, yet it reduces to a spherical polynomial result (which says slightly more than \cref{thm:polsph}) by approximating the ball by a cap of a sphere of large radius, then chopping the area between this cap and the opposite one in spherical segments of equal widths. The spherical case needs to be carefully applied to a polynomial of the form $P\cdot G_R$, where $R$ corresponds to the radius of the sphere and is chosen such that it is almost impossible to have a maximum of $|P\cdot G_R|$ on the sphere outside the cap. Then one needs to take the limit as $R$ goes to infinity, being careful about uniform convergence on compact subsets. We refer the reader to \cite{avoidingzeros} for the details.

Let us end this section with the following enticing conjecture on covering a sphere by planks instead of zones.
\begin{conjecture} \label{conj:polysp}
If $S^{d-1}$, $d\geq 4$, is covered by planks $P_1,\ldots,P_n$, then
$$
w(P_1)+\cdots+w(P_n)\geq 2.
$$
\end{conjecture}
It is unlikely that the previous polynomial approach will bear fruit with regards to this conjecture, since the arguments first reduced the problem to the $d=2$ case, which is false here (in fact, it is readily seen that a covering of $S^1$ by planks of arbitrarily small widths is possible).
But it is known to hold for $d=3$ \cite{moese} and for $d=4$ when there are at most three planks \cite{avoidingzeros}.

\subsection{Strengthening the plank theorems and plank problems in other settings}
Another natural problem to consider in discrete geometry is that of covering by arbitrary convex bodies, which can serve as a strengthening of the questions on covering by planks. The \textit{inradius} of $C$ is the radius of the largest ball contained in $C$ and is denoted by $r(C)$. It has been shown \cite{bezdek2d} that if the closed unit disk is partitioned by $2$-dimensional convex bodies $C_1,\ldots,C_n$ (sometimes called \textit{convex cells}), then 
\begin{equation} \label{eq:inradeq1}
\sum_{i=1}^n r(C_i) \geq 1.
\end{equation} 
Note that this result is stronger than Tarski's plank problem for the unit disk in the sense that even if you partition the disk into $m$ cells $C_i$ and for each cell pick the narrowest plank that contains it, the sum of those minimal widths is still at least $2$. Indeed, each such plank will have width greater than or equal to the diameter of the incircle of the cell, namely $2r(C_i)$, so that $\sum_{i=1}^n w_i\geq \sum_{i=1}^m2r(C_i)\geq 2$. The proof of \eqref{eq:inradeq1} is also reminiscent of the arguments of Tarski that were outlined in \cref{sec:partialsol}. 
If we consider coverings of the unit ball of Banach spaces by convex bodies, then these subadditivity of the inradii questions improve the plank theorems. For example, while the following result remains open for general Banach spaces, it has been shown by Kadets \cite{kadetsthm} (and much earlier by Ohmann \cite{ohmann2} in the planar case) that if $C$ is a convex body in a Hilbert space which is covered by convex bodies $C_1, C_2,\ldots$, then 
\begin{equation} \label{eq:inradeq}
\sum_{i=1}^\infty r(C_i\cap C)\geq r(C).
\end{equation} 
It is sufficient to prove the result when $C$ is the unit ball of the Hilbert space $H$, so we immediately see how it relates to Ball's symmetric plank theorem. It also relies primarily on the following generalization of Bang's lemma (which follows from an amalgamation of the results found in \cite{kadetsthm}).


\begin{proposition} \label{prop:kadets}
If $U_1,\ldots, U_n$ are finite sets of unit vectors in $\R^d$ such that $0$ is in the convex hull of each $U_i$, $(w_i)_{i=1}^n$ are nonnegative real numbers and $(v_i)_{i=1}^n$ are vectors in $\R^d$, then there exists $u_i\in U_i$ ($i=1,\ldots, n$) so that $u=\sum_{i=1}^n w_iu_i$ satisfies 
$$
\inner{u-v_i}{u_i} \geq w_i
$$
for all $1\leq i\leq n$.
\end{proposition}
To see how it generalizes \cref{lemBangV2}, if $u_1,\ldots,u_n$ are unit vectors in $\R^d$ and $m_i$ are real numbers, simply take $U_i=\{u_i,-u_i\}$ and pick $v_i$ so that $\inner{v_i}{u_i}=m_i$.

Bang's lemma and results like \eqref{eq:inradeq1} and \eqref{eq:inradeq} have been explored and generalized in other ways. For instance, there have been recent extensions to the notion of generalized nonconvex multi-planks \cite{multi} (which also apply Bogn\'ar's trick) and to contact pairs in the spirit of the Colorful Carath\'eodory Theorem \cite{ambrusbang}.
It is also known that an analogue of \eqref{eq:inradeq} holds for coverings of large balls in spherical space \cite{largeballs}, and that an analogue of \eqref{eq:inradeq1} holds for the \textit{inradius relative to a convex body} $C\subset \R^d$ \cite{kinradii}, which for a convex set $D\subset \R^d$ is defined by
$$
r_C(D) := \sup\{\lambda\geq 0:\: \exists\: x\in\R^d \text{ such that } \lambda C + x\subset D\}
$$
with $r_C(\emptyset)=-\infty$, once we replace arbitrary coverings by certain convex partitions. Here a partition is a covering by a family of closed convex sets with
disjoint interiors. Even for $\R^2$, 
there exist coverings that do not contain partitions, for instance planks going through the center of a disk such that their intersection with the boundary of the disk forms two disjoint arcs and all of these arcs partition the boundary, thus forming a ``sunflower''. The authors of \cite{kinradii} prove that the analogue of Kadets' theorem in $\R^2$ holds for any partition.

\begin{theorem} \label{thm:kadan}
    Let $C\subset \R^2$ be a convex body with convex partition $C_1\cup\cdots C_n=C$. Then
    $$
\sum_{i=1}^n r_C(C_i)\geq 1.
    $$
\end{theorem}
However, in higher dimensions, they need to restrict to convex partitions induced by inductive partitions of $\R^d$ (see \cite{kinradii}). 
The key idea is that any convex partition $C_1\cup\cdots\cup C_n=C\subset\R^2$ can be extended to a partition $V_1\cup\cdots\cup V_n$ of all of $\R^2$, so in this case the inductive approach always works: we consider translations of the $V_i$'s by a vector $y$ which we can pick such that one of the $(V_i+y)\cap C$ disappears, so we can remove it from the partition, then repeat the process on the partition into fewer parts.
Also note that analogues for infinite-dimensional Banach
spaces can be obtained via standard approximation arguments, and that while a version of Kadets' theorem holds in the context of spherical geometry,
it does not work for hyperbolic space.

Let us also mention that coverings by specific convex bodies have been investigated. For instance, a question on covering by cylinders was actually introduced by Bang \cite{bang1951}. He noted that while \cref{Thm Bang} means that the least $1$-dimensional projection of a convex body is not greater than the sum of the $1$-dimensional projections of its parts, the analogous statement for $2$-dimensional projections does not hold. K. Bezdek and Litvak \cite{lattice} were then motivated to study the following question. 
\begin{question} \label{qu:2}
Is the sum of the base areas of finitely many cylinders covering a $3$-dimensional convex body at least half of the minimum area of a $2$-dimensional projection of the body?
\end{question}
If true, this would be sharp as implied by an example of Bang \cite{bang1951}. While the answer to \cref{qu:2} is still unknown, the authors of \cite{lattice} were able to show that the sum in question is at least one-third of the minimum area of a $2$-dimensional projection of the convex body covered by cylinders.
On a related note, the following appealing special case of a conjecture of A. Bezdek \cite{annus} on the covering of annuli by planks remains open.
\begin{conjecture}
If a sufficiently small concentric circular hole is removed from a unit disk and the resulting annulus is covered by planks, then the total width of the covering is at least $1$.
\end{conjecture}
In other words, the planks that cover the annulus can be rearranged to cover the entire unit disk. One can see that the analogous statement for convex domains in the plane does not hold in general (it fails for an equilateral triangle, for example),
although it was stated in this form at first as a variant of Bang's problem. In \cite{polygon}, the polygons for which it holds were characterized, and recent work \cite{spiky} extended this to convex bodies in $\R^2$ and $\R^3$ having a minimal width direction in which it is ``spiky'' (as defined in \cite{polygon} and \cite{spiky}).

\subsubsection{Completing the picture in settings other than the Euclidean space} \label{sec:complete}
We have seen examples of plank theorems for general Banach and Hilbert spaces, but we can formulate plank problems in other settings or give more precise theorem statements in some cases. 
 
The reader is advised to compare the following $L_p(\mu)$ plank theorem from \cite{realBan} with \cref{Ball complexe}, where $L_p(\mu)$ stands for the set of measurable functions $f$ such that $\int |f|^p\,d\mu<\infty$.
\begin{theorem} \label{thm:Lp}
If $f_1,\ldots,f_n\in L_p(\mu)^*$ are unit functionals, then there is a point $x$ in the unit ball of $L_p(\mu)$ such that for $1\leq k\leq n$, 
$$
|f_k(x)|\geq 
\begin{cases}
       n^{-1/p}, &1\leq p\leq 2,\\
       n^{-1/q}, &2\leq p\leq \infty,
     \end{cases}
$$
where $q$ is the conjugate exponent of $p$, that is, $1/p+1/q=1$. 
\end{theorem}
Observe that for $p=1$ and $p=\infty$, we obtain $|f_k(x)|\geq 1/n$, which was already known from Ball's plank theorems. 
Let us look at the simple proof of \cref{thm:Lp} from \cite{realBan} to realize that this is, indeed, just an extension of the complex plank theorem.

\begin{proof}
We will need one fact about the \textit{Banach--Mazur distance} between normed spaces which are isomorphic as vector spaces,
$$
d(X,Y)= \inf\{\norm{T}\|{T^{-1}}\| \mid T:X\to Y \text{ is an isomorphism}\},
$$
 which is that if $E$ is a $d$-dimensional subspace of $L_p(\mu)$, then
$
d(E,\ell_2^d)\leq d^{|1/2-1/p|}
$
(see \cite{subLp} for the proof). 

As just observed, the cases $p=1$ and $p=\infty$ are dealt with already, and when $1<p<\infty$, $L_p(\mu)$ is reflexive so we can find $x_k$ in the unit ball of $L_p(\mu)$ for which $|f_k(x_k)|=\|f_k\|=1$ for all $1\leq k\leq n$. 
Now consider $E:=\text{span}\{x_1,\ldots,x_n\}$, which has dimension $d\leq n$. Then $\left.g_k:=f_k\right|_{E}$ is such that $\|g_k\|=1$ and
using our fact on the Banach--Mazur distance, we can find an isomorphism $T$ from $\ell_2^d$ to $E$ such that $\|T\|=1$ and $\|T^{-1}\|\leq d^{|1/2-1/p|}$. 

We are now in a position to apply \cref{Ball complexe}, which gives us a unit $x_0\in \ell_2^d$ for which $|(g_k\circ T)(x_0)|\geq \|g_k\circ T\|/\sqrt{n}$. 
Let $x:=Tx_0$. Then $\|x\|\leq \|T\|\|x_0\|=1$ and
$$
|f_k(x)|=|g_k(x)|=|g_k(Tx_0)| \geq \frac{\|g_k\circ T\|}{\sqrt{n}} \geq \frac{\|g_k\|/\|T^{-1}\|}{\sqrt{n}} = \frac{1}{n^{1/2}\|T^{-1}\|} \geq \frac{1}{n^{1/2}d^{|\frac{1}{2}-\frac{1}{p}|}}.
$$
This proves the claim.
\end{proof}

Now let us consider a discrete plank problem.
A \textit{lattice} of $\R^d$ is an additive subgroup of $\R^d$ that is isomorphic to $\Z^d$ and that spans $\R^d$. Equivalently, it corresponds to the set of all linear combinations with integer coefficients of the elements of a basis of $\R^d$. We say that a  \textit{convex set of lattice points} is the intersection of a convex body with a lattice. Extending plank problems to this discrete setting, it is natural to study coverings of these convex sets by specific subspaces like hyperplanes. A long-standing and shockingly simple to state conjecture in the area, due to Corzatt \cite{corzatt}, states that if a convex set of lattice points can be covered by $n$ lines, then these lines can be taken to have at most four different slopes (see \cref{fig:n} for an example). 

\begin{figure}[ht]
\centering
\begin{tikzpicture} 
  \draw (-0.1,0)--(4.1,0) ;
  \filldraw (0.2,0) circle (2pt) ;
    \filldraw (1.1,0) circle (2pt) ;
      \filldraw (2,0) circle (2pt) ;
      \filldraw (2.9,0) circle (2pt) ;
      \filldraw (3.8,0) circle (2pt) ;
  \draw (1.1,-1.2)--(1.1,3) ;
    \filldraw (1.1,-0.9) circle (2pt) ;
    \filldraw (1.1,0.9) circle (2pt) ;
    \filldraw (1.1,1.8) circle (2pt) ;
    \filldraw (1.1,2.7) circle (2pt) ;
    
    \filldraw (0.2,-0.9) circle (2pt) ;
    \filldraw (2,0.9) circle (2pt) ;
    \filldraw (2.9,0.9) circle (2pt) ;
    \filldraw (2.9,1.8) circle (2pt) ;
    \filldraw (2,1.8) circle (2pt) ;

  \draw (0,-1.1)--(3.1,2) ;
    \draw (4,-0.2)--(0.9,2.9) ;

\end{tikzpicture}
\caption{A set of points covered by lines with four different slopes.} \label{fig:n}
\end{figure}

This conjecture is reminiscent of the disk case ($n=2$) of Tarski's problem, which was proven first, but can be extended to higher dimensions as in Tarski's plank problem. Recall from the Introduction that in the covering of convex bodies in the plane, a parallel covering is optimal. You can think of this as saying that we only require one parallel class to cover the convex body. Then one can wonder more generally about the previous conjecture whether for all dimension $d$ and for all sets of convex lattice points in $\R^d$, their coverings can be taken to have at most a constant (depending on $d$) number of parallel classes. A close problem of Brass \cite{brass} asks to find the smallest constant $c(d)$ such that if a convex set of lattice points in $\R^d$ is covered by $n$ hyperplanes, then it can also be covered by $c(d)\cdot n$ parallel hyperplanes. 

Related to this is the notion of lattice width, as investigated first by K. Bezdek and Hausel \cite{bezhau94}. The \textit{lattice width} of a convex body in $\R^d$ with respect to the integer lattice $\Z^d$ is defined as 
$$
w_{\Z^d}(C)=\min_{y\in \Z^d\setminus\{0\}}\Bigl\{\max_{x\in C}\inner{x}{y} - \min_{x\in C}\inner{x}{y}\Bigr\}.
$$
This number is used to give a lower bound on the number of hyperplanes needed to cover $C\cap \Z^d$. It is conjectured that if $n$ hyperplanes are used, then $n\geq w_{\Z^d}(C)-d$. 
If true, this would be best possible \cite{bezhau94}. It is known that the following weaker result holds \cite{talata}.
\begin{theorem} \label{thm:talata}
If $C$ is a convex body in $\R^d$ covered by $n$ hyperplanes,
then
$$
n\geq c\cdot \frac{w_{\Z^d}(C)}{d}-d,
$$
where $c>0$ is an absolute constant.
\end{theorem}
Note that this result implies that the constant $c(d)$ alluded to in the
aforementioned problem of Brass is at most $c\cdot d^2$.

Furthermore, Ball's plank theorem has been used to improve \cref{thm:talata} for centrally symmetric convex bodies whose lattice width is at most quadratic in dimension \cite{lattice}. In particular, the argument relies on \cref{cor:pigeon} that will be presented in \cref{section approx}. The lower bound obtained is
$$
n\geq c\cdot \frac{w_{\Z^d}(C)}{d\ln(d+1)}.
$$

A similar problem asks to determine the minimal number of hyperplanes needed to cover the vertices of the hypercube $\{\pm 1\}^n$. Clearly, two hyperplanes suffice (e.g., given by $x_1=1$ and $x_1=-1$), but this is only a ``$1$-dimensional'' covering, and so we would like to ask that each variable $x_1,\ldots,x_n$ is used in the equations defining the hyperplanes covering the vertices. But this problem is still not interesting, because one could add a redundant equation, e.g., $x_1+\cdots + x_n=0$, to obtain a covering with $3$ hyperplanes. But this new hyperplane is not really needed, so it is natural to also ask that none of the hyperplanes in the covering are redundant. This leads to the notion of an \textit{essential cover} of the $n$-dimensional hypercube, introduced by Linial and
Radhakrishnan \cite{linrad}. If $e(n)$ stands for the smallest size of an essential cover of $\{\pm 1\}^n$, then they showed 
$$
\frac{\sqrt{4n+1}+1}{2} \leq e(n)\leq \Big\lceil \frac{n}{2} \Big\rceil + 1.
$$
The upper bound is still the best known to date, but the lower bound has been improved in a series of works by different authors which all use Bang's lemma in their proof. It was shown that the lower bound can be taken to be $C_1\cdot n^{0.52}$ in \cite{yeye}, $C_2\cdot n^{5/9}/(\log n)^{4/9}$ in \cite{araujo}, and $C_3\cdot n^{2/3}(\log n)^{2/3}$ in \cite{sxu}, for some positive constants $C_i$ and all $n$ large enough.

Of particular interest to some authors are \textit{inverse plank problems}, that is, how to guarantee that certain systems of planks (or more generally convex bodies) can cover a given convex set (see \cite{inversesurvey}, \cite{kuppach}, and \cite{makaipach}). We often talk of \textit{permitting a translative covering} of a convex set.
Some authors have also recently used techniques of symplectic geometry to study plank problems. In particular, they propose \textit{symplectic capacities} (invariants of symplectic manifolds that have already proven to be useful in problems at the intersection of convex and symplectic geometry) as a vehicle to reformulate plank problems and to pose a question stronger than the known plank problems. For the specific definitions and statements given by the authors, see the paper \cite{sympl} and the references therein.

It is important to note that we have not covered 
all of the types of plank problems that have been investigated in the literature. However, from our above exposition we can see that the proofs of the plank problems can be categorized based on the methods used. We distinguish a few main approaches as follows. A \textit{volumetric} approach such as the one used by Tarski (and Moese) in the proof of Tarski's plank problem for a disk in $\R^2$ and a sphere in $\R^3$; one based on \textit{optimization} arguments as in Bang's lemma and its generalizations; one using \textit{matrix} methods as in Ball's proof of the symmetric plank theorem and its complex counterpart; a \textit{polynomial} approach to plank problems as in \cref{sec:polynomialapproach}; and other ad hoc approaches, for example the \textit{deformation} argument that was used to prove an analogue of Kadets' theorem for certain convex partitions (see the discussion surrounding \cref{thm:kadan}).
Finally, consult \cite{classical, bezdek} or \cite[Section 3.4]{brass} to find more open problems and conjectures that are related to plank problems.






\section{Applications in analysis} \label{sec:app2}
We continue to survey applications of plank theorems in different fields by looking into various problems coming from analysis.

\subsection{Polarization problems and linear polarization constants} \label{sec:polarization}
The real polarization problem states that for any unit vectors $u_1,\ldots,u_n\in \R^d$, we can find a unit vector $v\in\R^d$ such that
\begin{equation} \label{eq:realpol}
\prod_{k=1}^n |\inner{u_k}{v}| \geq n^{-n/2},
\end{equation}
and that equality is attained only if the $u_k$ form an orthonormal system in $\R^n$ (see \cite{ambrus2019analytic}, \cite{realpolprob}).
Similar problems arise from the concept of linear polarization constants. Formally, for $X$ a Banach space, define the \textit{$n$th linear polarization constant of $X$} as 
\begin{align} \label{eq:nth}
c_n(X)&:=\inf\{M>0:\norm{\phi_1}\cdots\norm{\phi_n}\leq M \norm{\phi_1\cdots\phi_n} \text{ for all }\phi_1,\ldots,\phi_n\in X^*\} \nonumber \\
&=1/\inf_{\phi_1,\ldots,\phi_n\in S_{X^*}}\sup_{\|x\|=1}|\phi_1(x)\cdots \phi_n(x)|,
\end{align}
where $S_{X^*}$ is the unit sphere of $X^*$, the dual space of $X$. Note that $c_n(X)$ is an isometric invariant of $X$.
Here is one example of the $n$th linear polarization constant: it follows from \cref{thm:Lp} that
$$
c_n(L_p(\mu))\leq \begin{cases}
       n^{n/p}, &1\leq p\leq 2,\\
       n^{n/q}, &2\leq p\leq \infty.
     \end{cases}
$$
Furthermore, if $\dim L_p(\mu)\geq n$, then this result is sharp when $1\leq p\leq 2$ \cite{realBan}. 

Several authors have been interested in the \textit{polarization constant of $X$}, namely $$c(X):=\limsup_{n\to\infty} c_n(X)^{1/n}.$$ As it was shown in \cite{realBan}, the limit always exists and so we may replace $\limsup_{n\to\infty}$ by $\lim_{n\to\infty}$ in the definition of $c(X)$. 
It is natural to wonder what we can infer on $c(X)$ depending on $X$. At first glance, it is hard to say much more than $c(X)\geq 1$, so perhaps it is often the case that  $c(X)=\infty$. Fortunately, it has been proven that $c(X)=\infty$ if and only if dim$X=\infty$ \cite{realBan}. Most of the study of the polarization constant of $X$ relies on bounds for $c_n(X)$. For instance, it has been shown
that in complex Banach spaces, $c_n(X)\leq n^n$, which is best possible \cite{complexBan}. Indeed, the coordinate functionals 
\begin{align*}
    \phi_i\colon \ell_1^n(\C) &\to \C\\
    (z_1,\ldots,z_n)          &\mapsto z_i
\end{align*}
give 
$$
\norm{\phi_1}\cdots\norm{\phi_n} = n^n \norm{\phi_1\cdots\phi_n}.
$$
Surprisingly, it was proven a few years later that the same upper bound $c_n(X)\leq n^n$ holds true in real Banach spaces and is also best possible. It simply follows from Ball's plank theorem \ref{thm Ball} upon picking $m_i=0$ and $w_i=1/n$ in the statement of the theorem, as first observed in \cite{realBan}.

While it is inherently hard to calculate the $n$th polarization constant of specific Banach spaces, the study of $c_n(\ell_2^n(\mathbb{K}))$, where $\mathbb{K}=\R$ or $\C$, is of utmost importance, since from \cite{realBan} we know that 
\begin{equation} \label{eq:ell2}
c_n(\ell_2^n(\mathbb{K})) \leq c_n(X) \leq n^{n/2}c_n(\ell_2^n(\mathbb{K}))
\end{equation}
for all infinite-dimensional Banach spaces $X$. You can think of this lower bound as saying that Hilbert spaces have the smallest $n$th polarization constants among infinite-dimensional Banach spaces. The proof uses the concept of the Banach--Mazur distance between two isomorphic Banach spaces that we alluded to in \cref{sec:complete}.

The next step is to inquire what happens for Hilbert spaces, where we modify the definition of the polarization constant according to the Riesz representation theorem. We may also see that
$$
c_n(X)=\sup\{c_n(Y): Y \text{ is a closed subspace of } X, \dim Y\leq n\}
$$
when $X$ is a Hilbert space.
Observe that from Ball's complex plank theorem, there exists a unit vector $z$ such that $|\inner{v_j}{z}|\geq 1/\sqrt{n}$ for unit vectors $v_1,\ldots,v_n$ and $1\leq j\leq n$, whence 
\begin{equation} \label{eq:ballpol}
|\inner{v_1}{z}|\cdots|\inner{v_n}{z}| \geq n^{-n/2}.
\end{equation}
This is the complex version of the polarization problem presented at the beginning of the section, and comparing with \eqref{eq:nth}, we see that it implies $c_n(H)\leq n^{n/2}$ for $H$ a complex Hilbert space.
This was proven a few years prior to Ball's result by Arias-de-Reyna \cite{Hilcom} using results on permanents and Gaussian random variables, and it was also shown that $c_n(\ell_2^n(\mathbb{C})) = n^{n/2}$, so it follows from \eqref{eq:ell2} that $c_n(H) = n^{n/2}$ when $\dim H \geq n$. Note that by the nonhomogeneous polynomial complex plank theorem seen in \cref{sec:revisit}, we have a nonhomogeneous extension of \eqref{eq:ballpol} (this was stated in \cite{complnou} with a typo in the result): for unit vectors $v_1,\ldots,v_n$ and all vectors $y_1,\ldots,y_n$, there exists a unit vector $z$ such that
$$
|\inner{v_1}{z-y_1}|\cdots|\inner{v_n}{z-y_n}| \geq n^{-n/2}.
$$

Analogously to the Banach space case, it is conjectured that the same constant works for real Hilbert spaces, which was proven for $n\leq 5$ in \cite{2004} and recently for $n\leq 14$ in \cite{pinasco}. Similarly, it is also expected that $c_n(\ell_2^n(\R))=n^{n/2}$. Note that the real polarization problem presented at the beginning of this section would imply this conjecture.
We cannot expect to find a proof using known plank theorems, for even though Ball's complex plank theorem is sharp in complex Hilbert spaces, it does not hold for real Hilbert spaces (see \cref{sec:reforzone}).
Nevertheless, many authors have provided upper bounds for the polarization constant of real Hilbert spaces. A simple argument follows from the result on complex Hilbert spaces along with the complexification of a real Hilbert space $H$, yielding $c_n(H)\leq 2^{n/2-1}n^{n/2}$, and the best result is from \cite{pfaff} and gives
$$
c_n(H) < \Bigl(\frac{3\sqrt{3}}{e}n\Bigr)^{n/2} < (1.912n)^{n/2}.
$$

Another polarization problem has been investigated, deemed more interesting from a geometric point of view. It states that
given unit vectors $u_1,\ldots,u_n$ in $\R^d$, there exists a unit vector $v\in \R^d$ such that
\begin{equation} \label{eq:strongpol}
  \sum_{i=1}^n \frac{1}{\inner{u_i}{v}^2} \leq n^2.
\end{equation}
Note that this is stronger than the real polarization problem presented earlier because the arithmetic-geometric mean inequality yields
$$
\prod_{i=1}^n  \Bigl(\frac{1}{\inner{u_i}{v}^2}\Bigr)^{1/n} \leq \frac{1}{n}\sum_{i=1}^n \frac{1}{\inner{u_i}{v}^2} \leq n,
$$
which is equivalent to \eqref{eq:realpol}. It was shown in \cite{2013} that if $z_1,\ldots,z_n\in S^1$, then there exists $z\in S^1$ such that
$$
\sum_{k=1}^n \frac{1}{|z_k-z|^2} \leq \frac{n^2}{4},
$$ 
which settled the $d=2$ case of the previous polarization problem.
Indeed, thinking of this case as being formulated on the complex unit circle, we are given $u_1,\ldots,u_n\in S^1$, say ${u_k=(\cos(t_k/2),\sin(t_k/2))}$, and we are looking for $v\in S^1$ such that \eqref{eq:strongpol} holds. Writing 
$$v=\big(\cos\frac{t-\pi}{2}, \sin\frac{t-\pi}{2}\big),$$
$z=e^{it}$, and $z_k=e^{it_k}$, then $\inner{u_k}{v}$ reduces to $|z-z_k|/2$ after using trigonometric identities. These ideas have been generalized and formal definitions as well as several related results on polarization can be found in \cite[Chapter 14]{springerpola}.

We also remark that both of the polarization problems presented in this section can be reformulated in terms of inverse eigenvectors, which were introduced briefly in \cref{sec:reforzone}, and this connects them more deeply with the complex plank problem. The ideas that appear in the proofs of polarization problems are also connected with the polynomial approach presented in \cref{sec:polynomialapproach}, as already mentioned there.

\subsection{Applications in functional analysis} 
Let us first mention that plank theorems can be interpreted as a strengthening or extensions of well-known theorems.
For example, \cref{refor} is a reformulation of Bang's conjecture and can be interpreted as a geometric version of the pigeonhole principle. Thus, \cref{thmboth} is in some sense a geometric pigeonhole principle for convex symmetric bodies. Many other examples come from functional analysis.

\subsubsection{Classical results in functional analysis}

Plank theorems have connections with fundamental theorems in functional analysis.
 Recall that the Hahn--Banach theorem roughly says that linear functionals on  a linear subspace can be extended to the whole normed space while preserving the norm. Using the \textit{weak}-$^*$ topology of $X^*$ (the smallest topology such that every element of $X$ is continuous with respect to that topology), we can ``inverse'' \cref{thm Ball infini} to obtain a ``multiple'' extension of Hahn--Banach, as observed in \cite{ball1991}. 

\begin{cor} \label{cor:ball}
Let $(x_i)_{i=1}^\infty$ be unit vectors in a normed space $X$, $(m_i)_{i=1}^\infty$ real numbers and $(w_i)_{i=1}^\infty$ nonnegative real numbers such that $\sum_{i=1}^\infty w_i \leq 1$. Then there exists a linear functional $\phi\in X^*$ with norm at most $1$ such that 
$$
\abs{\phi(x_i)-m_i}\geq w_i
$$
for all $i$.
\end{cor}



We can also obtain a quantitative improvement of the uniform boundedness principle, or Banach--Steinhaus theorem, which we recall here for the reader's convenience.
\begin{theorem} \label{thm bs}
Let $X$ be a Banach space, $Y$ a normed space, and $\Phi$ a family of bounded linear operators from $X$ to $Y$. If
$$
\sup_{\phi\in\Phi}\norm{\phi(x)}<\infty
$$
for all $x\in X$, then
$$
\sup_{\phi\in\Phi}\norm{\phi}<\infty.
$$
\end{theorem}

It is often useful to consider the contrapositive statement: if a family of bounded linear operators from a Banach space to a normed space is not uniformly bounded, then it is not pointwise bounded. Suppose that we have a family of linear operators on $X^*$ that are not uniformly bounded, where $X$ is a Banach space. Then their operator norm is not bounded and, in particular, the sum $\sum_{n=1}^\infty n\norm{\phi_n}^{-1}$ will be smaller than $1$ for some $\phi_1,\phi_2,\ldots$ in this family.
But then, \cref{thm Ball infini} with $m_i=0$ gives a vector $x$ in the unit ball of $X$ such that
$$
\Bigl|\frac{\phi_n}{\norm{\phi_n}}(x)\Bigr| > n\norm{\phi_n}^{-1}
$$
for all $n$. Rearranging, we get $\abs{\phi_n(x)}>n$ for all $n$, and so our operators are not pointwise bounded. We can push this reasoning much further (see \cref{sec:ubp}).

\subsubsection{Plank theorems and linear dynamics}
Recall the \textit{weak topology} on $X$ is the smallest topology such that every element of the dual space $X^*$ is continuous with respect to that topology. We can define weak versions of usual topological notions if they hold with respect to the weak topology.
We say a bounded linear operator $T$ on $X$ is \textit{hypercyclic} if there exists $x\in X$ such that the orbits $T^n(x)=(T\circ\cdots\circ T)(x)$ are dense in $X$, and \textit{supercyclic} if $\{\alpha T^n(x):\alpha\in\C,\;n\geq 0\}$ is dense in $X$. If these sets are dense in the weak topology, then we talk about \textit{weak hypercyclicity} and \textit{supercyclicity}.

Plank theorems have many connections with linear dynamics. For instance, Ball's theorems helped in obtaining results on weak hypercyclicity and supercyclicity of operators on Banach and Hilbert spaces, and more generally in answering questions related to the weak topology on Banach spaces (see \cite{weak3, weak1, weak2}). Knowing that on infinite-dimensional Banach spaces, the weak and the norm topology do not coincide, here is one simple example of this kind of question.

\begin{question}
Let $(x_n)_{n\in \N}$ be a sequence in an infinite-dimensional Banach space $X$. Under which condition on $x_n$ is the set $\{x_n : n \in \N\}$ weakly dense in $X$?
\end{question}

In general, one can prove that if $x_n$ is such that $\|x_n\|$ tends to infinity fast enough, then ${\{x_n : n \in \N\}}$ is weakly closed in $X$, so certainly not weakly dense.
Here is a more precise sharp formulation of this idea. The proofs we follow are given in \cite[Propositions 5.2--5.4]{weak2} and \cite[Section 10.1.1]{weak3}.

\begin{thm} \label{thm:weak}
Let $(x_n)_{n\in \N}$ be a sequence of nonzero vectors in the Banach space $X$. If
$$
\sum_{n=1}^\infty \frac{1}{\|x_n\|}<\infty,
$$
then the set $\{x_n : n \in \N\}$ is weakly closed in $X$. Furthermore, if $X$ is a real or complex Hilbert space, then it is sufficient that $\sum_{n=1}^\infty \|x_n\|^{-2}<\infty$.
\end{thm}

\begin{proof}
We only prove the Banach space case here. The proof of the complex Hilbert one is virtually the same, and for the real Hilbert space, one needs to consider the complexification of $X$. It will be enough to show that the weak closure of the $x_n$ does not contain $0$ (to see why, replace $x_n$ with $x_n-z$, where $z\in X\setminus \{x_n : n\in \N\}$). 

Let $s:=\sum_{n=1}^\infty\|x_n\|^{-1}$ and $\alpha_n:=1/s\|x_n\|$, so that $\sum_{n=1}^\infty \alpha_n=1$. Then \cref{cor:ball} with the unit vectors $x_n/\|x_n\|$ yields a $\phi$ in $X^*$ such that
 $$\abs{\phi(x_n)}\geq \frac{1}{s}$$ for all $n\in\N$. It follows that $0$ is not in the weak closure of $\{x_n : n\in \N\}$.
 \end{proof}

\cref{thm:weak} is tight in the sense that it has been shown that if $X$ is an infinite-dimensional Banach space and $(c_n)_{n\in\N}$ is a sequence of positive numbers such that $\sum_{n=1}^\infty c_n^{-2}=\infty$, then there exists a sequence $(x_n)_{n\in\N}\in X$ such that $\|x_n\|=c_n$ for all $n$ and $0$ is in the weak closure of $\{x_n:n\in \N\}$ (see \cite[Proposition 6.11]{weak2}).

\subsubsection{A stronger version of the uniform boundedness principle} \label{sec:ubp}
While Theorem \ref{thm bs} is no longer true if we omit the supremum, one may wonder under what conditions it holds that there is a sequence of bounded linear operators $\phi_n\in\Phi$ and some $x\in X$ such that $\|\phi_n(x)\|\to\infty$. This question is intimately related to linear dynamics via the study of orbits of linear operators, hence the first partial results came from searching for conditions ensuring that orbits are going to infinity (see \cite{beauart} and the book of Beauzamy \cite[Chapter III]{beauzabook}, in particular III.2.C.1 for the following theorem). 

\begin{theorem} \label{thm:beau}
Let $T$ be an operator on a Hilbert space $H$ such that 
$$
\sum_{n=1}^\infty \frac{1}{\|T^n\|}<\infty.
$$
Then there is a dense set of points $x\in H$ such that $\|T^nx\|\to\infty$ as $n\to\infty$.
\end{theorem}
This theorem is not sharp. In fact, it suffices that $\sum_{n=1}^\infty c_n(T^n)^2<\infty$ for the result to hold, where 
$$
c_n(T^n):=\inf\{\|T^n|_{Z}\|:Z\subset X,\: \codim Z<n\}
$$
is the $n$th \textit{Gelfand number} of the operator $T^n$ (see III.2.C.5 in \cite{beauzabook}).
The hypotheses of \cref{thm:beau} are reminiscent of \cref{thm:weak} and
the proof relies on Bang's lemma. 
It was noted in $1990$ by Beauzamy that a solution to the symmetric plank problem (that is, to Bang's conjecture for symmetric convex bodies) would extend \cref{thm:beau} more generally to a sequence of bounded operators and thus strengthen the uniform boundedness principle. This actually prompted Ball to work on plank problems.
We can find a proof of this stenghtening of \cref{thm:beau} in a paper of Müller and Vr\v{s}ovsk\'{y} \cite{orbits} for a single $x\in H$, where the hypotheses for a complex Hilbert space are slightly weakened. It was later showed \cite{m-v} that the conditions in the following theorem are sufficient to ensure that the set of $x\in X$ such that $\norm{T_nx}\to\infty$ is dense in $X$. 


\begin{thm} \label{thm:m-v}
Let $T_n$, $n\in \N$, be a sequence of bounded linear operators on a Banach space $X$. If 
$$
\sum_{n=1}^\infty\frac{1}{\norm{T_n}}<\infty,
$$
then there exists $x\in X$ such that $\norm{T_nx}\to\infty$. Furthermore, if $X$ is a complex Hilbert space, then it is sufficient that $\sum_{n=1}^\infty\norm{T_n}^{-2}<\infty$.
\end{thm}

\begin{proof} We treat the case where $X$ is a complex Hilbert space because the proof is more natural in this setting. A few modifications are needed in the Banach space case (see \cite{orbits}). 

Choose $\beta_n$ positive real numbers going to infinity such that
$$
s:=\sum_{n=1}^\infty\frac{\beta_n}{\norm{T_n}^2}<\infty,
$$
and define $$\alpha_n:=\frac{1}{\sqrt{s+1}}\frac{\sqrt{\beta_n}}{\norm{T_n}}.$$ For all $n$, take $g_n\in X^*$ such that $\norm{g_n}\leq 1$ and $\norm{T_n^*g_n}\geq \norm{T_n^*}/2=\norm{T_n}/2 $. Finally, define
$$
f_n:=\frac{T_n^*g_n}{\norm{T_n^*g_n}} \in X^*.
$$
We can apply \cref{Ball complexe} with these linear functionals. Indeed, since the coefficients $\alpha_n$ satisfy $$\sum_{n=1}^\infty \alpha_n^2 = s(s+1)^{-1}<1,$$
we know there exists 
$x\in X$ with $\norm{x}= 1$ and
$\abs{\inner{x}{f_n}}\geq \alpha_n$ for all $n$. Then, 
$$
\norm{T_nx}\geq \norm{T_nx}\norm{g_n} \geq \abs{\inner{T_nx}{g_n}} = \abs{\inner{x}{T_n^*g_n}} = \norm{T_n^*g_n} \abs{\inner{x}{f_n}} 
\geq \frac{\norm{T_n}}{2}\alpha_n.
$$
This expression simplifies to $\sqrt{\beta_n}/2\sqrt{s+1}$ and thus tends to $\infty$ with $n$ by our choice of $\beta_n$.
\end{proof}

We end this section by noting that \cref{thm:m-v} is optimal
in the sense that there exists a Banach space $X$ and operators $T_n$ such that $$\sum_{n=1}^\infty \frac{1}{\norm{T_n}^{1+\varepsilon}}<\infty,$$
yet there is no $x\in X$ with $\norm{T_n x}\to\infty$. An explicit construction is given in \cite{orbits}.

\subsection{Applications in harmonic analysis}
The Riemann--Lebesgue lemma says that the Fourier coefficients of an $L_1(\mathbb{T})$ function tend to $0$ as $\abs{n}\to\infty$, where $\mathbb{T}$ stands for the unit circle. It is natural to wonder how fast this happens. In particular, can this decay be arbitrarily slow? Kolmogorov answered this question positively in 1923. More precisely, he showed that for any choice of sequence of positive integers $(a_n)_{n\in \Z}$ with $a_n\to 0$ as $\abs{n}\to\infty$, there exists a function $f\in L_1(\mathbb{T})$ such that $\widehat{f}(n)\geq a_n$ for all $n$. A similar problem, asking whether all square summable sequences are dominated by the Fourier coefficients of some continuous function on the unit circle $\mathbb{T}$, turned out to be much harder to settle. Fortunately, the mathematicians de Leeuw, Kahane, and Katznelson \cite{KKL} were able to provide a positive answer in 1977. They even showed that we may take this function to be bounded on $\mathbb{T}$. You can think of their result as saying that we cannot distinguish functions in $L_2(\mathbb{T})$ from functions in $C(\mathbb{T})$ if we only look at the size of their Fourier coefficients.

\begin{thm} \label{kkl}
Let $(a_n)_{n\in \Z}\in \ell_2(\Z)$. Then there exists $f\in C(\mathbb{T})$ and an absolute constant $A>0$ such that 
  $ \norm{f}_\infty \leq A\left(\sum_{n\in \Z} \abs{a_n}^2\right)^{1/2}$  and
  $$ |\widehat{f}(n)| \geq \abs{a_n}$$ for all $n$.
\end{thm}
What is surprising with the proof of this theorem is that it considers sums of the form
$$
\sum_{i=1}^\infty \epsilon_i a_i \psi_i,
$$
where $(\psi_n)_{n=1}^\infty$ is a sequence of functions and $(\epsilon_n)_{n=1}^\infty$ is a random choice of signs, which is reminiscent of Bang's lemma once again. See the beginning of the proof of \cref{naz} below for an explicit example of how this kind of sum might appear.

We can go further with these ideas and ask: Given a sequence of functions $\psi_n\in L_1(\mathbb{T})$, under which conditions can we guarantee that there exists a bounded function $f$ such that 
$$
\abs{\inner{f}{\psi_n}}>\abs{a_n}
$$
for any sequence $(a_n)_{n=1}^\infty \in \ell_2(\Z)$? Unfortunately, the de Leeuw--Kahane--Katznelson theorem does not answer this question in general. However, many years later, a theorem of Nazarov solved this problem and it was proven using ideas coming from plank theorems. This is self-evident upon reading the title of Nazarov's paper: \textit{The Bang solution of the coefficient problem} \cite{Nazarov}.
In fact, according to Nazarov himself, the modification in the demonstration of Bang's result needed to prove the following theorem was so minor that he claimed, ``[...] I actually even do not pretend to be an author of the next two sections; rather I act there like a shadow that enters and goes over many strange places which completely eliminate the attention of his master just passing by.'' We remark that Nazarov's original article is in Russian, but an English translation is available \cite{Nazarov}.

\begin{thm} \label{naz}
Let $T$ be a probability space with measure $\mu$ and $(\psi_n)_{n=1}^\infty$ a sequence of unit functionals in $T^*$ such that
\begin{equation} \label{eq:2norm}
\Bigl\|\sum_{i=1}^\infty c_i\psi_i\Bigr\|_{2} \leq \Bigl(\sum_{i=1}^\infty c_i^2\Bigr)^{1/2}
\end{equation}
for all sequences of real coefficients $(c_n)_{n=1}^\infty$. For $2\leq p\leq \infty$ and $q$ its conjugate exponent, suppose that
\begin{equation} \label{eq:beta}
\norm{\psi_n}_q \geq \beta > 0
\end{equation}
for all $n$. Let $a_i$ be positive numbers such that $\sum_{i=1}^\infty a_i^2=1$. Then there exists $F\in L_p(T)$ such that
$$
\norm{F}_{p} \leq \Bigl(\frac{3\pi}{2}\Bigr)^{1-2/p}\beta^{-2}
$$
and, for all $n$,
$$
\abs{\inner{F}{\psi_n}}\geq a_n.
$$
\end{thm}

\begin{proof}
For a choice of signs $\epsilon=(\epsilon_i)_{i=1}^\infty$, let
$$
f_{\epsilon} := \sum_{i=1}^\infty \epsilon_i a_i \psi_i,
$$
and note that from \eqref{eq:2norm} and the assumption $\sum_{i=1}^\infty a_i^2=1$, it follows that  $f_{\epsilon}\in L_2(T)$ (more precisely, $\norm{f_\epsilon}_2\leq 1$ for all $\epsilon$).
Theorems of existence and uniqueness of ODEs allow us to define the (unique) function $\Phi$ by the initial conditions $\Phi(0)=\Phi'(0)=0$ and the differential equation
$$
\Phi''(x) = (1+x^2)^{2/p-1}.
$$

From the definition of $\Phi$, the integral $I(f):=\int_T \Phi(f)\,d\mu$ is well-defined and continuous in $L_2(T)$. Moreover, since $\{f_{\epsilon}\}$ is compact in the topology of $L_2(T)$, the integral attains a maximum for some $f_{\overline{\epsilon}}$. 
The $F$ we pick will be a rescaling of $\Phi'(f_{\overline{\epsilon}})$, so let us start by showing that $\Phi'(f_{\overline{\epsilon}}) \in L_p(T)$. By the definition of $\Phi(x)$ and Hölder's inequality,
\begin{align*}
\abs{\Phi'(x)} = \Bigl|\int_0^x (1+t^2)^{2/p-1}\,dt\Bigr| \leq \int_0^{|x|} (1+t^2)^{2/p-1}\,dt 
\leq \Bigl(\int_0^{|x|}\,dt\Bigr)^{2/p} \Bigl(\int_0^{|x|}(1+t^2)^{-1}\,dt\Bigr)^{1-2/p}.
\end{align*}
These integrals evaluate to $|x|^{2/p}(\arctan |x|)^{1-2/p}\leq |x|^{2/p}(\pi/2)^{1-2/p}$, so that
$$
\norm{\Phi'(f_{\overline{\epsilon}})}_p \leq \left(\frac{\pi}{2 }\right)^{1-2/p}\Bigl(\int_T|f_{\overline{\epsilon}}|^2\,d\mu\Bigr)^{1/p} \leq \left(\frac{\pi}{2 }\right)^{1-2/p},
$$
using our previous observation that $\norm{f_{\epsilon}}_2\leq 1$ for all $\epsilon$. Furthermore, we see that upon taking $F:=3^{1-2/p}\beta^{-2}\Phi'(f_{\overline{\epsilon}})$, we have
$$
\norm{F}_{p} \leq \Bigl(\frac{3\pi}{2}\Bigr)^{1-2/p}\beta^{-2},
$$
as desired.

Now we want to show that for all $j$,
\begin{equation} \label{eq:inner}
a_j\leq \abs{\inner{F}{\psi_j}} = 3^{1-2/p}\beta^{-2}\Bigl|\int_T \Phi'(f_{\overline{\epsilon}})\psi_j\,d\mu\Bigr|.
\end{equation}
For each $j$, let $f_j:=f_{\overline{\epsilon}} - 2\overline{\epsilon}_ja_j\psi_j$, that is, flip the sign of the $j$th term in the sum defining $f_{\overline{\epsilon}}$. 
Since $f_{\overline{\epsilon}}$ maximizes $I(\cdot)$, we have $0\leq \int_T (\Phi(f_{\overline{\epsilon}})-\Phi(f_j))\;d\mu$ for each $j$, and
by the mean value theorem, there exists a function $g$ between $f_j$ and $f_{\overline{\epsilon}}$ such that 
\begin{align*}
\int_T (\Phi(f_{\overline{\epsilon}})-\Phi(f_j))\;d\mu &=\int_T \Bigl(\Phi'(f_{\overline{\epsilon}})(f_{\overline{\epsilon}} - f_j) + \frac{1}{2}  \Phi''(g)(f_{\overline{\epsilon}} - f_j)^2\Bigr)\,d\mu  \\
&= 2a_j\int_T \Phi'(f_{\overline{\epsilon}})\overline{\epsilon}_j\psi_j + 2a_j^2\int_T \Phi''(g)\psi_j^2\,d\mu.
\end{align*}
Since $f_{\overline{\epsilon}} - f_j = 2\overline{\epsilon}_ja_j\psi_j$, we have
$$
0\leq 2a_j\int_T \Phi'(f_{\overline{\epsilon}})\overline{\epsilon}_j\psi_j + 2a_j^2\int_T \Phi''(g)\psi_j^2\,d\mu,
$$
whence
$$
\Bigl|\int_T\Phi'(f_{\overline{\epsilon}})\psi_j\;d\mu\Bigr| \geq a_j\int_T \Phi''(g)\psi_j^2\;d\mu.
$$
These previous steps are reminiscent of the strategy adopted in the proof of Bang's lemma.

We see from \eqref{eq:inner} that we will be done if we can show that 
$
\int_T \Phi''(g)\psi_j^2\;d\mu \geq 3^{2/p-1}\beta^2 .
$
Since $\frac{q}{2}+(1-\frac{2}{p})\frac{q}{2}=1$, Hölder's inequality gives
$$
 \Bigl(\int_T (1+g^2)\;d\mu\Bigr)^{1-q/2}\Bigl(\int_T (1+g^2)^{2/p-1}\psi_j^2\;d\mu \Bigr)^{q/2} \geq \int_T \abs{\psi_j}^q \geq \beta^q,
$$
where we used \eqref{eq:beta} to obtain the last inequality. Since $g$ lies between $f_j$ and $f_{\overline{\epsilon}}$, we have $g^2\leq f_j^2+f_{\overline{\epsilon}}^2$ and thus $\int_T (1+g^2)\;d\mu \leq \int_T (1+f_j^2+f_{\overline{\epsilon}}^2)\;d\mu \leq 3$, using that $\norm{f_{\epsilon}}_2\leq 1$ for all $\epsilon$. We conclude that 
$$
\int_T \Phi''(g)\psi_j^2\;d\mu = \int_T (1+g^2)\psi_j^2\;d\mu \geq 3^{1-2/q}\beta^2 = 3^{2/p-1}\beta^2. \qedhere
$$ 
\end{proof}

The following corollary of Nazarov's theorem gives a more direct answer to the coefficient problem. It follows from the $p=\infty$ case of \cref{naz} and an application of a weak form of Grothendieck's inequality (see \cite{ballconvex}, a paper of Ball that talks about Nazarov's theorem, among other things).

\begin{corollary}
Let $(\psi_n)_{n=1}^\infty$ be functions of norm $1$ in $L_1(\mathbb{T})$ such that there exists a constant $M>0$ with
$$
\Bigl\|\sum_{i=1}^\infty c_i\psi_i\Bigr\|_{1} \leq M\Bigl(\sum_{i=1}^\infty c_i^2\Bigr)^{1/2}
$$
for all sequences of real coefficients $(c_n)_{n=1}^\infty$. 
Then, for all $(a_n)_{n=1}^\infty\in \ell_2(\Z)$, there exists a bounded function $f$ such that
$$
\abs{\inner{f}{\psi_n}} > |a_n|
$$
for all $n$.
\end{corollary}

\subsubsection{Applications and extensions of Nazarov's result}
The coefficient problem is intimately related to applied problems via the Fourier transform, and so Nazarov's result has recently attracted attention in signal processing and electrical engineering. For instance, an important mathematical problem in computational imaging is to develop optimal coded apertures. In \cite{ajjanagadde2019near}, their fundamental limits were characterized using \cref{naz}, and a greedy algorithm that relies on the proof of this theorem was proposed. Also see \cite{signaldesign} for a discussion of various problems in signal processing where \cref{naz} might come in handy. In particular, observe that \cref{naz} does not require orthogonality of the unit functionals, but only asks for an $\ell_2$ estimate, hence it can be used in problems regarding frames and bases in more generality.

Both the results of de Leeuw, Kahane, and Katznelson, and Nazarov, give insight on the spectral structure of large sets in additive combinatorics, as observed for example in works of Green \cite{bg}. Here, for a function $f:\Z/N\Z\to \C$ and $r\in \Z/N\Z$, the \textit{Fourier transform} of $f$ at $r$ is defined as
$$
\widehat{f}(r)=\sum_x f(x)e(rx/N),
$$
 where $e(x)=\exp(2\pi ix)$.
As mentioned by Green, the result of Nazarov implies that the only information we can obtain on the large spectrum of a large subset of $\Z/N\Z$ comes from Parseval's theorem (via Chang's theorem, see \cite{bg} for more details and the following result).

\begin{theorem}
Let $\alpha_r$, $r\in \Z/N\Z$, be positive reals satisfying ${\sum_r \alpha_r^2\leq N/1600}$. Then there is a function $f:\Z/N\Z\to [0,1]$ such that $|f|=\sum_x f(x)= N/2$, and so that
$$
|\widehat{f}(r)|\geq \alpha_r|f|
$$
for all $r\in \Z/N\Z$.
\end{theorem}

Finally, we want to highlight the fact that soon after Nazarov's result, Lust-Piquard \cite{FLP} obtained noncommutative versions of the de Leeuw--Kahane--Katznelson and Nazarov theorems. In particular, the reader is invited to compare \cref{kkl} with the following statement, where we say that a matrix $A$ is in $\ell_\infty(\ell_2)$ if 
$
\|A\|_{\ell_\infty(\ell_2)}^2=\sup_{i}\sum_{j} \abs{a_{ij}}^2 < \infty.
$

\begin{thm} \label{thm:flp}
Let $A$ be such that $A$ and $A^*$ are in $\ell_\infty(\ell_2)$. Then there exists a matrix $B$ that defines a bounded operator from $\ell_2$ to $\ell_2$ such that
$$\norm{B}_{\ell_2\to\ell_2}\leq \sqrt{6} \max\{\norm{A}_{\ell_\infty(\ell_2)}, \norm{A^*}_{\ell_\infty(\ell_2)}\}$$ and
$
\abs{b_{ij}}\geq \abs{a_{ij}}
$
for all $i,j\geq 0$. Moreover, the constant $\sqrt{6}$ is best possible.
\end{thm}

In the previous result, it is possible to replace $\ell_2\to\ell_2$ with other spaces $\ell_p\to\ell_{q}$ modulo minor modifications.

It is interesting to note that this theorem was a central element in a new proof \cite{normsofschur} of a well-known result of Varopoulos on a characterization of \textit{Schur multipliers}, which are bounded operators from $\ell_2\to\ell_2$ that act by entrywise multiplication with a matrix. The following is a restatement of Varopoulos's result to resemble \cref{thm:flp}.

\begin{theorem}
Let $\mathscr{S}(A)$ be the set of symbols with matrix entries $b_{ij}$ satisfying $|b_{ij}|\leq |a_{ij}|$. Then the following are equivalent: 
\begin{enumerate}
    \item every symbol in $\mathscr{S}(A)$ defines a Schur multiplier;
    \item we can write $A=C+D$ where $C\in \ell_\infty(\ell_2)$ and $D^*\in \ell_\infty(\ell_2)$;
    \item for all finite sets of indices $I$ and $J$,
    $$
    \sum_{i\in I}\sum_{j\in J} |a_{ij}|^2 =O(|I|+|J|).
    $$
\end{enumerate}
\end{theorem}


\section{Applications in number theory} \label{sec:app3}
In this section, we present some applications of plank theorems in number theory.
It is interesting to note that alternatively, number theory can sometimes give insight into plank problems. For instance, letting the \textit{height} of a point $(x_1,\cdots,x_n)$ be $\max_{1\leq i\leq n}|x_i|$, Fukshansky \cite{fuk} used his bounds on the height of integral points of small height
outside a hypersurface defined by a nonzero integer polynomial over $\mathbb{Q}$
to obtain a lower bound for a discrete plank problem which concerns coverings of the set of integer lattice points in $\R^d$ contained in a cube by sublattices of rank $d-1$.
See \cref{sec:complete} for a discussion on discrete analogues of plank problems.

\subsection{Simultaneous Diophantine approximation and piercing convex bodies} \label{section approx}
Recall the classical result of Dirichlet in Diophantine approximation which states that for all real $\theta$, there exist infinitely many natural numbers $q$ such that
$$
\norm{q\theta} \leq \frac{1}{q},
$$
where $\norm{\cdot}$ denotes the distance to the nearest integer.
We say that $\theta$ is \textit{badly approximable} if the right-hand side of the previous inequality cannot be improved by any positive constant, that is, if there exists $c=c(\theta)>0$ such that $\norm{q\theta}>cq^{-1}$ for all $q\in \N$. The set of badly approximable numbers has Lebesgue measure zero, but full Hausdorff dimension. These numbers are also closely related to Littlewood's conjecture, a famous and long-standing open problem in Diophantine approximation, which says that if $\theta, \varphi \in \R$, then
$$
\liminf_{n\to\infty} n\norm{n\theta}\norm{n\varphi} = 0.
$$
Since there is also a \textit{simultaneous} Diophantine approximation theorem which states that for any $\theta, \varphi\in \R$, there exist infinitely many $q\in \N$ such that
$$
\max\{\norm{q\theta}^2,\norm{q\varphi}^2\} \leq \frac{1}{q},
$$
it is natural to define a \textit{pair of badly approximable numbers} as a pair of reals $\theta,\varphi$ such that there exists a positive constant $c=c(\theta,\varphi)$ which verifies
$$
\max\{\norm{q\theta }^2,\norm{q\varphi }^2\}>\frac{c}{q}
$$
for all $q\in\N$.
Davenport showed that the set of these pairs is uncountable, and it was later shown that in fact, it has Hausdorff dimension two. Davenport \cite{Davenport} also proved the following related theorem on simultaneous Diophantine approximation.

\begin{thm}
Let $(\lambda_i)_{i=1}^n$ and $(\mu_i)_{i=1}^n$ be real numbers. Then there exist $\alpha$ and $\beta$ such that $\alpha+\lambda_i$ and $\beta+\mu_i$ 
form a pair of badly approximable numbers for  $i=1,2,\ldots,n$.  Furthermore, we may take the constant $c=2^{-4n-7}$ in the definition of a pair of badly approximable numbers.
\end{thm}

This result has connections with convex geometry since Davenport's solution uses geometric arguments. More precisely, he proves by induction the existence of a square in the plane such that all points $(\alpha, \beta)$ in that square satisfy the requirements of the theorem. The generalization to higher dimensions and to simultaneous Diophantine approximation of many reals is immediate.
We can also infer from this result another type of geometric pigeonhole principle
\cite{Alexander}.

\begin{thm}
Let $C$ be a cube in $\R^d$ and $(H_i)_{i=1}^n$ hyperplanes. Then we can find in $C$ another cube having the same orientation, being at least $2^{-n}$ times as wide as $C$, and such that none of the $H_i$ intersect its interior.
\end{thm}

As mentioned in \cref{sec gen Bang}, Alexander already observed that a positive answer to Bang's conjecture would improve the previous theorem. As such, \cref{thm Ball} also improves it and, in particular, implies that for a symmetric convex body $C$, we can find a set that does not intersect any of the hyperplanes $H_i$ via a well-chosen dilation and translation of $C$ \cite{ball1991}.

\begin{cor} \label{cor:pigeon}
Let $C$ be a symmetric convex body in $\R^d$ and $(H_i)_{i=1}^n$ hyperplanes. Then there exists a vector $x$ such that $x+\frac{1}{n+1}C$ is in $C$ and such that none of the $H_i$ intersect the interior of $x+\frac{1}{n+1}C$.
\end{cor}

We follow the proof from \cite{ball1991}.

\begin{proof}
Without loss of generality, we take $C$ centered at the origin. Recall that $C$ being a symmetric convex body, it is the unit ball of some finite-dimensional normed space, so it makes sense to define $X$ such that $C$ corresponds to its unit ball. Consider the unit functionals $\phi_i\in X^*$ and the real numbers $m_i$ that correspond to the hyperplanes
$$
H_i=\{x\in \R^d : \phi_i(x)=m_i\}.
$$
\cref{thm Ball} gives 
$x'\in C$ such that
$$\Bigl|\phi_i(x')-\frac{n+1}{n}m_i\Bigr|\geq 1/n.$$
Changing the scale by a factor of $n/(n+1)$, we have an $\displaystyle x\in\frac{n}{n+1} C$, explicitly $\displaystyle x=\frac{n}{n+1}x'$, such that
\begin{equation} \label{eq:inverse}
\abs{\phi_i(x)-m_i}\geq \frac{1}{n+1}
\end{equation}
for all $i$, and $\displaystyle x+\frac{1}{n+1}C\subseteq C$. 
We just have to show that $
\displaystyle\text{int}\bigl(x+\frac{1}{n+1}C\bigr)$  will not be sliced by any of the $H_i$. For $\displaystyle y\in \text{int}\bigl(x+\frac{1}{n+1}C\bigr)$, we have $\norm{y-x}\leq (n+1)^{-1}$, whence
$$
\abs{\phi_i(y)-m_i-(\phi_i(x)-m_i)}=\abs{\phi_i(y-x)}\leq \norm{\phi_i}\norm{y-x}\leq \frac{1}{n+1}.
$$
Combining with \eqref{eq:inverse}, we see that $\phi_i(x)-m_i$ and $\phi_i(y)-m_i$ have the same (nonzero) sign, and in particular $\phi_i(y)\neq m_i$. The claim follows.
\end{proof}

\cref{cor:pigeon} is intimately related to a recent study on intersecting and piercing all the cells of the $n\times n$ chessboard $Q_n$ embedded in $[-1,1]^2$. These cells are explicitly given by
$$
c_{ij}=\biggl[-1+(i-1)\cdot \frac{2}{n}, -1+i\cdot\frac{2}{n}\biggr]\times \biggl[-1+(j-1)\cdot \frac{2}{n}, -1+j\cdot\frac{2}{n}\biggr]
$$
for $i,j\in \{1,\ldots,n\}$.
Here by a line $\ell$ \textit{intersecting} (respectively, \textit{piercing}) a cell $c_{ij}$, we mean that $\ell\cap c_{ij}\neq \emptyset$ (respectively, $\ell\cap \text{int}\: c_{ij}\neq \emptyset$). The authors of \cite{chess} conjectured that the minimum number of lines needed to pierce all the cells of $Q_n$ (call it $p_n$) is $n-1$ for all $n\geq 3$, which can be seen as a strengthening of the planar case of the previous corollary for lattice cells. By the same token, it would be an improvement of Ball's plank theorem for integer points. While only the bounds $p_n>0.7n$ for $n$ sufficiently large and $p_n\leq n-1$ for $n\geq 3$ were obtained, studying the minimum number of lines needed to intersect all the cells of $Q_n$ is easier but still related to the plank theorem of Ball. In fact, it is an essential ingredient in proving the following \cite{chess}.  

\begin{proposition} \label{prop:chess}
For all $n\geq 1$, the minimum number of lines needed to intersect all the cells of $Q_n$ is $\lceil \frac{n}{2}\rceil$.
\end{proposition}

\begin{proof}
Write $h_n$ for the minimum number of lines needed to intersect all the cells of $Q_n$. By picking every second line separating the columns of $Q_n$, it is immediate that $h_n\leq \lceil\frac{n}{2}\rceil$ (\cref{fig:chess} serves as a visual aid).

\begin{figure}[ht]
\centering
\begin{tikzpicture}[scale=1,
line cap=round,
line join=round,
]
\fill[line width=1pt,fill=black,fill opacity=1] (1,0) -- (1,1) -- (2,1) -- (2,0) -- cycle;
\fill[line width=1pt,fill=black,fill opacity=1] (0,2) -- (0,1) -- (1,1) -- (1,2) -- cycle;
\fill[line width=1pt,fill=black,fill opacity=1] (2,1) -- (2,2) -- (3,2) -- (3,1) -- cycle;
\fill[line width=1pt,fill=black,fill opacity=1] (1,2) -- (1,3) -- (2,3) -- (2,2) -- cycle;
\fill[line width=1pt,fill=black,fill opacity=1] (5.5,3) -- (5.5,4) -- (6.5,4) -- (6.5,3) -- cycle;
\fill[line width=1pt,fill=black,fill opacity=1] (7.5,3) -- (7.5,4) -- (8.5,4) -- (8.5,3) -- cycle;
\fill[line width=1pt,fill=black,fill opacity=1] (4.5,2) -- (4.5,3) -- (5.5,3) -- (5.5,2) -- cycle;
\fill[line width=1pt,fill=black,fill opacity=1] (6.5,2) -- (6.5,3) -- (7.5,3) -- (7.5,2) -- cycle;
\fill[line width=1pt,fill=black,fill opacity=1] (5.5,1) -- (5.5,2) -- (6.5,2) -- (6.5,1) -- cycle;
\fill[line width=1pt,fill=black,fill opacity=1] (7.5,1) -- (7.5,2) -- (8.5,2) -- (8.5,1) -- cycle;
\fill[line width=1pt,fill=black,fill opacity=1] (4.5,0) -- (4.5,1) -- (5.5,1) -- (5.5,0) -- cycle;
\fill[line width=1pt,fill=black,fill opacity=1] (6.5,0) -- (6.5,1) -- (7.5,1) -- (7.5,0) -- cycle;
\draw [line width=1pt] (1,0)-- (1,1);
\draw [line width=1pt] (1,1)-- (2,1);
\draw [line width=1pt] (2,1)-- (2,0);
\draw [line width=1pt] (2,0)-- (1,0);
\draw [line width=1pt] (0,2)-- (0,1);
\draw [line width=1pt] (0,1)-- (1,1);
\draw [line width=1pt] (1,1)-- (1,2);
\draw [line width=1pt] (1,2)-- (0,2);
\draw [line width=1pt] (2,1)-- (2,2);
\draw [line width=1pt] (2,2)-- (3,2);
\draw [line width=1pt] (3,2)-- (3,1);
\draw [line width=1pt] (3,1)-- (2,1);
\draw [line width=1pt] (1,2)-- (1,3);
\draw [line width=1pt] (1,3)-- (2,3);
\draw [line width=1pt] (2,3)-- (2,2);
\draw [line width=1pt] (2,2)-- (1,2);
\draw [line width=1pt] (0,0)-- (0,3);
\draw [line width=1pt] (0,3)-- (3,3);
\draw [line width=1pt] (3,3)-- (3,0);
\draw [line width=1pt] (3,0)-- (0,0);
\draw [line width=2pt,color=ffqqqq] (3,3.5)-- (3,-0.5);
\draw [line width=2pt,color=ffqqqq] (1,3.5)-- (1,-0.5);

\draw [line width=1pt] (4.5,0)-- (8.5,0);
\draw [line width=1pt] (4.5,0)-- (4.5,4);
\draw [line width=1pt] (4.5,4)-- (8.5,4);
\draw [line width=1pt] (8.5,4)-- (8.5,0);
\draw [line width=1pt] (5.5,3)-- (5.5,4);
\draw [line width=1pt] (5.5,4)-- (6.5,4);
\draw [line width=1pt] (6.5,4)-- (6.5,3);
\draw [line width=1pt] (6.5,3)-- (5.5,3);
\draw [line width=1pt] (7.5,3)-- (7.5,4);
\draw [line width=1pt] (7.5,4)-- (8.5,4);
\draw [line width=1pt] (8.5,4)-- (8.5,3);
\draw [line width=1pt] (8.5,3)-- (7.5,3);
\draw [line width=1pt] (4.5,2)-- (4.5,3);
\draw [line width=1pt] (4.5,3)-- (5.5,3);
\draw [line width=1pt] (5.5,3)-- (5.5,2);
\draw [line width=1pt] (5.5,2)-- (4.5,2);
\draw [line width=1pt] (6.5,2)-- (6.5,3);
\draw [line width=1pt] (6.5,3)-- (7.5,3);
\draw [line width=1pt] (7.5,3)-- (7.5,2);
\draw [line width=1pt] (7.5,2)-- (6.5,2);
\draw [line width=1pt] (5.5,1)-- (5.5,2);
\draw [line width=1pt] (5.5,2)-- (6.5,2);
\draw [line width=1pt] (6.5,2)-- (6.5,1);
\draw [line width=1pt] (6.5,1)-- (5.5,1);
\draw [line width=1pt] (7.5,1)-- (7.5,2);
\draw [line width=1pt] (7.5,2)-- (8.5,2);
\draw [line width=1pt] (8.5,2)-- (8.5,1);
\draw [line width=1pt] (8.5,1)-- (7.5,1);
\draw [line width=1pt] (4.5,0)-- (4.5,1);
\draw [line width=1pt] (4.5,1)-- (5.5,1);
\draw [line width=1pt] (5.5,1)-- (5.5,0);
\draw [line width=1pt] (5.5,0)-- (4.5,0);
\draw [line width=1pt] (6.5,0)-- (6.5,1);
\draw [line width=1pt] (6.5,1)-- (7.5,1);
\draw [line width=1pt] (7.5,1)-- (7.5,0);
\draw [line width=1pt] (7.5,0)-- (6.5,0);
\draw [line width=2pt,color=ffqqqq] (7.5,4.5)-- (7.5,-0.5);
\draw [line width=2pt,color=ffqqqq] (5.5,4.5)-- (5.5,-0.5);
\end{tikzpicture}
\caption{Picking every second line separating the columns of $Q_n$ implies that $h_n\leq \lceil \frac{n}{2}\rceil$.} \label{fig:chess}
\end{figure}

For the lower bound, suppose for a contradiction that the lines $\ell_1,\ldots,\ell_{\lceil\frac{n}{2}\rceil-1}$ intersect all the cells of $Q_n$. Call $u_i$ the unit normal vector to $\ell_i$ and let $X$ be $\R^2$ endowed with the $\ell_\infty$-norm. For each nonzero $u\in \R^2$, define a unit linear functional $\phi_u$ by
$$
\phi_u(x)= \frac{\inner{x}{u}}{\|u\|_1}.
$$
Then we have an explicit description of the lines (hyperplanes) $\ell_i$ as
$$
\ell_i = \{x\in \R^2:\phi_{u_i}(x)=t_i\}
$$
for some $t_i\in \R$. We apply Ball's \cref{thm Ball} with equal widths $2/n+\epsilon$, where $\epsilon>0$ is chosen such that $$\Bigl(\Bigl\lceil\frac{n}{2}\Bigr\rceil-1\Bigr)\Bigl(\frac{2}{n}+\epsilon\Bigr)=1,$$ thus satisfying the requirement of the theorem. It follows that there is some $x$ in the unit ball of $X$ (namely, in $[-1,1]^2$) with 
\begin{equation} \label{eq:balleps}
|\phi_{u_i}(x)-t_i|\geq \frac{2}{n}+\epsilon
\end{equation}
for each $i$. Since 
$|\phi_{u_i}(x)-t_i|\leq w$ is equivalent to $x$ being in the union of all closed squares with edge length $2w$ and center on $\ell_i$,
we see that \eqref{eq:balleps} means that there is an open square centered in $[-1,1]^2$ and with side lengths strictly greater than $4/n$ that is not intersected by any of the $\ell_i$, yet every such square must contain a cell of $Q_n$ (recall the definition of $c_{ij}$). This is a contradiction.
\end{proof}

We remark that a higher dimensional analogue of \cref{prop:chess} holds, and with virtually the same proof, in the sense that the number of hyperplanes needed to intersect each cell of the $d$-dimensional box $Q_n^d$ is $\lceil\frac{n}{2}\rceil$ (see \cite{chess}).

\subsection{Sphere packings}
Recall from \cref{sec:complete} the definition of a lattice of $\R^n$.
Geometrically, a lattice can be thought of as a regular paving of Euclidean space. The study of lattices connects convex geometry and number theory. For instance, Ball \cite{ballball} saw a link between sphere packings and Bang's lemma after working on his plank theorems. This allowed him to obtain the asymptotically best lower bound on the density of sphere packings in $\R^n$ at the time. This has subsequently been improved by a constant factor in large dimensions \cite{venk} and, very recently, by a factor of $\log n$ \cite{impballs} and of $n$ \cite{newbd}. But what is of particular interest to us is his novel approach rather than the exact expression for the lower bound itself, so we shall describe it now.

Many of the upper bounds on sphere packing densities are obtained via a \textit{lattice packing in $\R^n$}, namely a family of balls of the same radius in $\R^n$ with centers on the coordinates of a lattice and with disjoint interiors. Its \textit{density} is the proportion of space covered by this lattice packing. A classical result of Minkowski (see \cite{hlawka}) says that for all $n$, we can find a sphere packing in $\R^n$ with density at least $2^{1-n}\zeta(n)$, where $\zeta(n)=\sum_{k=1}^\infty k^{-n}$ is the Riemann zeta function, but you can think of $\zeta(n)$ as some term tending to $1$ from above with $n$. The best result for a long time was from Davenport and Rogers \cite{DR} and indicated a lower bound of roughly $1,68\cdot n2^{-n}$.
Here is Ball's improvement to this lower bound.

\begin{thm} \label{emp}
For all $n\geq 1$, there exists a sphere packing in $\R^n$ having density greater than or equal to 
$2(n-1)2^{-n}\zeta(n)$.
\end{thm}
 
It is straightforward to compare the three previous lower bounds for different values of $n$. Note that Ball did not fully justify the apparition of  $\zeta(n)$ in his result, because this term appears naturally after using Möbius inversion (which can be thought of as the identity $\sum_{d\mid n}\mu(d)=1$ when $n=1$ and $0$ otherwise) and the formula $\sum_{k=1}^\infty \mu(k)k^{-n}=\zeta(n)^{-1}$ (it is, in fact, the same argument that Minkowski used). Here, the Möbius function $\mu(n)$ is supported on squarefree integers, and on these integers it is defined by $\mu(1)=1$ and $\mu(n)=(-1)^r$ if $n$ has $r$ distinct prime factors.

We present here the arguments of Ball that make use of the ideas of plank theorems and we only sketch the other arguments. A few remarks will first allow us to simplify the proof.

\begin{rem} \label{balls}
The \textit{determinant} of a lattice is an invariant defined as the absolute value of the determinant of the matrix representation of the linear isomorphism between this lattice and $\Z^n$. In particular, the density of a sphere packing of volume $V$ in $L$ is $V/\det L$. Moreover, we say that a lattice $L$ is \textit{admissible} with respect to a symmetric convex body $C$ if the only point of $L$ that the interior of $C$ contains is $0$. Take $C$ to be the ball of radius $2r$ centered at $0$. Then a lattice $L$ contains a sphere packing of balls of radius $r$ if and only if $L$ is admissible with respect to $C$.
Without loss of generality, we will be looking for a lattice with determinant $1$, so that its density coincides with the volume of the balls in the sphere packing. Hence, Ball's result with follow if we find a determinant $1$ lattice that is admissible for the ball of volume $2(n-1)$ centered at $0$. But by a compactness argument, it suffices to find a family of determinant $1$ lattices that are admissible for the ball of volume $\geq 2(n-1)-o(1)$. 
Equivalently, we can fix the lattice to be
$\Z^n$ and find origin-centered admissible ellipsoids of volume
\begin{equation} \label{eq:vol}
\text{Vol}(E_R)\geq 2(n-1) - o(1)
\end{equation}
as $R\to\infty$. Here, by \textit{origin-centered ellipsoids}, we mean deformations of the standard unit ball $B_1=\{x\in \mathbb{R}^n: \|x\|^2\leq 1\}$ under a linear transformation. The volume of an ellipsoid $E=L(B_1)$ is $\text{Vol}(E)=\sqrt{\det Q}\cdot\text{Vol}(B_1)$, where $Q=LL^T$.
\end{rem}

\begin{proof}[Proof of \cref{emp}]
Define the ellipsoids $E_R$ mentioned in \cref{balls} by the equation
\begin{equation} \label{ellipse}
\inner{u}{x}^2+\norm{x}^2\leq R^2
\end{equation}
for $x\in \R^n$ and $u=u(R)$ to be chosen to satisfy the inequality \eqref{eq:vol}. From our \cref{balls}, the volume of $E_R$ is simply 
\begin{equation} \label{volER}
\text{Vol}(E_R)=\frac{1}{\sqrt{1+K^2}}\text{Vol}(B_R),
\end{equation}
where $K=K(R)$ is the length of $u$ and $B_R$ the ball of radius $R$ centered at the origin. 
Moreover $\sqrt{1+K^2}=K+o(1)$ as $R\to\infty$, since then the length of $u$ will also tend to infinity.
It will thus suffice to show that
\begin{equation} \label{vol}
\frac{K}{\text{Vol}(B_R)}\leq \frac{1}{2(n-1)}+o(1)
\end{equation}
as $R\to\infty$.

By the definition of an admissible lattice, $\Z^n$ is admissible for $E_R$ if 
\begin{equation} \label{adm}
  \inner{u}{z}^2+\norm{z}^2\geq R^2
\end{equation}
for all $z\in \Z^n\setminus \{0\}$. Moreover, it is clear that if  $\norm{z}\geq R$, then \eqref{adm} holds. Hence it is enough to consider the $z$ such that $0<\norm{z}<R$, in which case \eqref{adm} holds if and only if
\begin{equation} \label{adapte}
\Biggl|\inner{u}{\frac{z}{\norm{z}}}\Biggr|\geq \Bigl(\frac{R^2}{\norm{z}^2}-1\Bigr)^{1/2}.
\end{equation}
Since the right-hand side of equation \eqref{adapte} (call it $\theta_z$) is then nonnegative and $z/\norm{z}$ is unitary, we know by \cref{lemBangV2} (Bang's lemma) that we can find a $u$ satisfying \eqref{adapte} of the form 
$$
u(R)=\frac{1}{2}\sum_{0<\norm{z}<R}\epsilon_z \theta_z \frac{z}{\norm{z}},
$$
where $\epsilon_z \in \{\pm 1\}$ (the $1/2$ factor comes from the fact that we apply Bang's lemma with both choices of $\pm z/\norm{z}$ for every $z$ such that $0<\norm{z}<R$). Now that we know that $\Z^n$ is admissible for $E_R$ with this choice of $u(R)$, we have to show that  \eqref{vol} holds. Since the only unknown in this equation is now $K$, we will try to estimate this value. We have
\begin{equation}
    K=\norm{u}=\sqrt{\inner{u}{u}}=\inner{u}{v}
\end{equation}
where $v=u/\norm{u}$. By definition and linearity, this expression is
$$
\frac{1}{2}\sum_{0<\norm{z}<R}\epsilon_z \frac{\inner{z}{v}}{\norm{z}} \theta_z \leq \frac{1}{2}\sum_{0<\norm{z}<R} \frac{\abs{\inner{z}{v}}}{\norm{z}} \Bigl(\frac{R^2}{\norm{z}^2}-1\Bigr)^{1/2}.
$$
Call this last expression $K'$. 

The rest of the proof takes us too far afield from the usual plank problems arguments, but let us mention that Ball establishes that $K'/R^n$ approximates a certain integral over the unit ball, and that in particular
$$
\frac{K'}{\text{Vol}(B_R)} 
\to \frac{n}{2} \Bigl(\int_{0}^{\pi/2}\sin^{n-2}t\, dt \Bigr)^{-1}\int_0^{\pi/2} \cos t\sin^{n-2}t \,dt \int_{0}^1 (1-r^2)^{1/2}r^{n-2}\, dr.
$$
With an appropriate change of variables, we find
$$
\int_{0}^1 (1-r^2)^{1/2}r^{n-2}\, dr = \frac{1}{n}\int_{0}^{\pi/2}\sin^{n-2}t\, dt,
$$
and so
$$
\frac{K'}{\text{Vol}(B_R)} 
\to \frac{1}{2} \int_0^{\pi/2} \cos t\sin^{n-2}t \,dt.
$$
This last integral can be computed directly and equals $(n-1)^{-1}$, proving \eqref{vol}.
\end{proof}




\section*{Acknowledgments}
The author would like to thank Thomas Ransford for reviewing earlier versions of this work and providing plenty of helpful comments which helped improve the quality of this survey. The author is also grateful to Gergely Ambrus and Alexandr Polyanskii for suggesting many new references, and Keith Ball for pointing out the work of Beauzamy that is presented in \cref{sec:ubp}. Finally, the author thanks the referees for their helpful comments and suggestions, especially for bringing forth recent work and references.


\bibliographystyle{abbrv}
\bibliography{mathbib.bib}

\bigskip

 \textsc{William Verreault, 
Department of Mathematics,  
University of Toronto,    
Toronto, ON, M5S 2E4, Canada} \par\nopagebreak
  \textit{E-mail address:} \texttt{william.verreault@utoronto.ca}

\end{document}